\documentclass[12pt]{amsart}
\usepackage{amsmath,amsthm,amsfonts,amssymb}

\usepackage{hyperref}
\usepackage[all]{xy}
\usepackage{graphicx}
\usepackage{amscd}
\usepackage{graphicx}
\usepackage{pb-diagram}
\usepackage{color}
\usepackage{enumerate}

\numberwithin{equation}{section}

\theoremstyle{plain} 
\newtheorem{theorem}{Theorem}[section]

\newtheorem{lemma}[theorem]{Lemma}
\newtheorem{prop}[theorem]{Proposition}

\theoremstyle{definition}
\newtheorem{definition}[theorem]{Definition}

\newtheorem{conjecture}[theorem]{Conjecture}

\theoremstyle{remark}
\newtheorem{remark}[theorem]{Remark}

\newcommand{\CP}{\mathbb{C}P}

\newcommand{\N}{\mathbb{N}}
\newcommand{\R}{\mathbb{R}}
\newcommand{\C}{\mathbb{C}}
\newcommand{\Z}{\mathbb{Z}}
\newcommand{\Q}{\mathbb{Q}}
\newcommand{\PP}{\mathbb{P}}

\newcommand{\bz}{\boldsymbol{z}}
\newcommand{\bm}{\boldsymbol{m}}
\newcommand{\bn}{\boldsymbol{n}}

\newcommand{\bt}{\boldsymbol{t}}
\newcommand{\PT}{\mathbb{P}^1_{3,3,3}}
\newcommand{\PH}{\mathbb{P}^1_{2,3,6}}
\newcommand{\PQ}{\mathbb{P}^1_{2,4,4}}

\newcommand{\dbar}{\overline{\partial}}

\newcommand{\CM}{\mathcal{M}}

\newcommand{\WT}[1]{\widetilde{#1}}
\newcommand{\OL}[1]{\overline{#1}}

\newcommand{\CO}{\mathcal O}

\newcommand{\da}[2]{\Delta^{{#1}/{#2}}_{1}}
\newcommand{\db}[2]{\Delta^{{#1}/{#2}}_{2}}
\newcommand{\dc}[2]{\Delta^{{#1}/{#2}}_{3}}
\newcommand{\ta}[2]{t_{1, \frac{#1}{#2}}}
\newcommand{\tb}[2]{t_{2, \frac{#1}{#2}}}
\newcommand{\tc}[2]{t_{3, \frac{#1}{#2}}}
\newcommand{\pt}{[\rm{pt}]}

\begin{document}
\author[Hong]{Hansol Hong}
\address{Department of Mathematical Sciences\\ Seoul National University\\ Gwanak-ro 1\\ Gwanak-gu \\Seoul 151-747\\ Korea}
\email{hansol84@snu.ac.kr, hansol84@gmail.com}
\author[Shin]{Hyung-Seok Shin}
\address{Department of Mathematical Sciences\\ Seoul National University\\ Gwanak-ro 1\\ Gwanak-gu \\Seoul 151-747\\ Korea}
\email{shs83@snu.ac.kr}

\title{On quantum cohomology ring of elliptic $\mathbb{P}^1$ orbifolds }
\maketitle
\begin{abstract}
We compute  quantum cohomology rings of elliptic $\mathbb{P}^1$ orbifolds via orbi-curve counting.
The main technique is the classification theorem which relates holomorphic orbi-curves with
certain orbifold coverings. The countings of orbi-curves are related to the integer solutions of Diophantine equations.
This reproduces the computation of Satake and Takahashi in the case of  $\PP^1_{3,3,3}$ via different method.
\end{abstract}

\tableofcontents

\section{Introduction}

The theory of holomorphic curves has been a great tool to understand the geometry of a symplectic manifold. The quantum cohomology ring plays an important role in Hamiltonian dynamics and mirror symmetry. Quantum cohomology counts J-holomorphic spheres inside a symplectic manifold which intersect three given (co)cycles. As the counting also includes constant holomorphic spheres, quantum cohomology deforms the classical cup product on the singular cohomology ring.

Later, Chen and Ruan \cite{CR2} defined the quantum cohomology for symplectic orbifolds. Their orbifold quantum cohomology ring captures stringy phenomenon in the sense that it also includes twisted sectors of a given orbifold as well as usual cocycles on the underlying space of the orbifold. In order to do this, one also allows the domain curves to have oribfold singularities. That is, one should count holomorphic {\em orbi-spheres} as well.

The main objects which we will study throughout the paper are one of simplest types of such orbifolds, spheres with three cone points called the {\em orbifold projective lines}. The orbifold quantum cohomology of these spheres indeed have a richer structure than that of the ordinary smooth sphere, since one additionally considers interaction among three singular points through J-holomorphic orbi-spheres as mentioned. We shall briefly review the orbifold quantum cohomology in Section \ref{sec:GWtheoryorb}.

Orbifold projective lines of our interest in this paper are those which admit elliptic curves as their manifold covers, and have three singular points $z_1$, $z_2$ and $z_3$. There are three such orbifold projective lines, $\mathbb{P}^1_{3,3,3}$, $\mathbb{P}^1_{2,3,6}$ and $\mathbb{P}^1_{2,4,4}$, where subindices indicate the orders of singularities at three orbifold points $z_i$. We will call these {\em elliptic orbifold projective lines}. (We remark that there is in fact one more orbifold projective line which is a quotient of elliptic curve, that is the orbifold projective line $\mathbb{P}^1_{2,2,2,2}$ with four $\Z_2$-singular points, which will not be considered in this paper.)

Recently, Satake-Takahashi \cite{ST} computed the full genus-0 Gromov-Witten potential for $\mathbb{P}^1_{3,3,3}$ and $\mathbb{P}^1_{2,2,2,2}$ making use of algebraic argument (for e.g. WDVV-equations). Furthermore, Krawitz-Shen \cite{KS} independently computed the potentials for $\PT$, $\PH$ and $\PQ$ even for all genera and proved the Landau-Ginzburg/Calabi-Yau correspondences for these examples.

We would like to mention that our method is different from theirs. Namely, our purpose here is to reproduce the (small) quantum product terms of the potential by directly counting holomorphic orbi-spheres. For this we will classify all holomorphic orbi-spheres with three markings in Section \ref{subsec:orbsphorbcover}. Interestingly, we find that these orbi-spheres have an one-to-one correspondence with the solutions of certain {\em Diophantine equations} depending on the lattice structures on the universal covers of orbifold projective lines constructed from the preimages of three singular points. 


The quantum product on $\PT$ is related to the Diophantine equation $Q_F (a,b):=a^2 -ab + b^2=d$.

\begin{theorem}\label{thm:main1}
For $\PT$, the only nontrivial 3-fold Gromov-Witten invariants are
\begin{align}
\label{eq:mainthmPT123} \langle \Delta_1^{1/3}, \Delta_2^{1/3}, \Delta_3^{1/3} \rangle_{0,3}^{\PT} \\
\label{eq:mainthmPT111}
\langle \Delta_i^{1/3}, \Delta_i^{1/3}, \Delta_i^{1/3} \rangle_{0,3}^{\PT}
\end{align}
for $i=1,2,3$, where subindices of $\Delta$ indicate the singular points $z_i$. If one denotes the moduli space of degree-$d$ holomorphic orbi-spheres contributing to \eqref{eq:mainthmPT123} by $\mathcal{M}_{0,3,d} (\PT;\Delta_1^{1/3}, \Delta_1^{1/3}, \Delta_1^{1/3})$ and that for \eqref{eq:mainthmPT111} by $\mathcal{M}_{0,3,d} (\PT;\Delta_i^{1/3}, \Delta_i^{1/3}, \Delta_i^{1/3})$, then 
$$\# \mathcal{M}_{0,3,d} (\PT;\Delta_1^{1/3}, \Delta_2^{1/3}, \Delta_3^{1/3}) = \frac{1}{6}\# \{ (a,b) : Q_F (a,b) =d, \,\,d \equiv 1 \,\,({\rm mod} \,3) \}$$
$$\# \mathcal{M}_{0,3,d} (\PT;\Delta_1^{1/3}, \Delta_1^{1/3}, \Delta_1^{1/3}) = \frac{1}{3}\# \{ (a,b) : Q_F (a,b) =d, \,\, d \equiv 0 \,\, ({\rm mod} \,3) \} $$
\end{theorem}

Similarly, solutions the Diophantine equation $Q_G (a,b):=a^2 + b^2=d$ are assigned to holomorphic orbi-spheres in $\PQ$.

\begin{theorem}\label{thm:main2}
For $\PQ$, the nontrivial contributions to 3-fold Gromov-Witten invariants come only from the moduli spaces
$$\mathcal{M}_{0,3,d} (\PQ; \da{1}{2}, \Delta_j^{1/4}, \Delta_k^{1/4} ), \mathcal{M}_{0,3,d} (\PQ; \Delta_j^{2/4}, \Delta_j^{1/4}, \Delta_j^{1/4} ), \mathcal{M}_{0,3,d} (\PQ; \Delta_j^{2/4}, \Delta_k^{1/4}, \Delta_k^{1/4})$$
 for $j, k = 2, 3$ and
$$\# \mathcal{M}_{0,3,d} (\PQ; \da{1}{2}, \Delta_j^{1/4}, \Delta_k^{1/4} ) = \frac{1}{4}\# \{ (a,b) : Q_G (a,b) =d, \,\,d \equiv 1 \,\,({\rm mod} \,4) \}$$
$$\# \mathcal{M}_{0,3,d} (\PQ; \Delta_j^{2/4}, \Delta_j^{1/4}, \Delta_j^{1/4} ) = \frac{1}{4}\# \{ (a,b) : Q_G (a,b) =d, \,\, d \equiv 0 \,\, ({\rm mod} \,4) \} $$
$$\#\mathcal{M}_{0,3,d} (\PQ; \Delta_j^{2/4}, \Delta_k^{1/4}, \Delta_k^{1/4}) = \frac{1}{4}\# \{ (a,b) : Q_G (a,b) =d, \,\,d \equiv 2 \,\,({\rm mod} \,4) \}$$
\end{theorem}

The quantum product on $\PH$ is also related to the Diophantine equation $Q_F (a,b):=a^2 + b^2=d$, but now we consider $d$ modulo $6$.

\begin{prop}\label{prop:main}
For $\PH$, we also have a similar statement related to the number of solutions of $Q_F(a,b)=d$ considering $d \,\, ({\rm mod}\, 6)$ and $d \,\, ({\rm mod}\, 3)$. Nontrivial 3-fold Gromov-Witten invariants are listed as follows.
\begin{enumerate}
\item
$$\# \mathcal{M}_{0,3,d} (\PH; \da{1}{2}, \db{1}{3}, \dc{1}{6} ) = \frac{1}{6}\# \{ (a,b) : Q_F (a,b) =d, \,\,d \equiv 1 \,\,({\rm mod} \,6) \},$$
$$\# \mathcal{M}_{0,3,d} (\PH; \dc{3}{6}, \db{1}{3}, \dc{1}{6} ) = \frac{1}{6}\# \{ (a,b) : Q_F (a,b) =d, \,\,d \equiv 4 \,\,({\rm mod} \,6) \},$$
$$\# \mathcal{M}_{0,3,d} (\PH; \da{1}{2}, \dc{2}{6}, \dc{1}{6} ) = \frac{1}{6}\# \{ (a,b) : Q_F (a,b) =d, \,\,d \equiv 3 \,\,({\rm mod} \,6) \},$$
$$\# \mathcal{M}_{0,3,d} (\PH; \dc{3}{6}, \dc{2}{6}, \dc{1}{6} ) = \frac{1}{6}\# \{ (a,b) : Q_F (a,b) =d, \,\,d \equiv 0 \,\,({\rm mod} \,6) \}.$$
\item
$$\# \mathcal{M}_{0,3,2d} (\PH; \dc{2}{6}, \dc{2}{6}, \dc{2}{6} ) = \frac{1}{6}\# \{ (a,b) : Q_F (a,b) =d, \,\,d \equiv 0 \,\,({\rm mod} \,3) \},$$
$$\# \mathcal{M}_{0,3,2d} (\PH; \db{1}{3}, \db{1}{3}, \dc{2}{6} ) = \frac{1}{6}\# \{ (a,b) : Q_F (a,b) =d, \,\,d \equiv 1 \,\,({\rm mod} \,3) \},$$
$$\# \mathcal{M}_{0,3,2d} (\PH; \db{1}{3}, \db{1}{3}, \db{1}{3} ) = \frac{1}{3}\# \{ (a,b) : Q_F (a,b) =d, \,\,d \equiv 0 \,\,({\rm mod} \,3) \}.$$
\end{enumerate}
In the item (2), there are only even-degree holomorphic orbi-spheres.
\end{prop}

In addition to these, there are two more 3-fold Gromov-Witten invariants
$\langle \dc{1}{6}, \dc{1}{6}, \dc{4}{6} \rangle^{\PH}_{0,3}$ and $\langle \db{2}{3}, \dc{1}{6}, \dc{1}{6} \rangle^{\PH}_{0,3}$, for which we only give a heuristic counting in Conjecture \ref{conj:P1366}. We are not able to associate solutions of a Diophantine equation to these holomorphic orbi-spheres since their domain is $\PP^1_{3,6,6}$ which does not admit an elliptic curve as a covering unlike the other cases.

%



On the other hand, closed-string mirror symmetry for $\mathbb{P}^1_{a,b,c}$ has been intensively studied in many preceding works. For example, Gromov-Witten theory on $\mathbb{P}^1_{a,b,c}$ with $\frac{1}{a} + \frac{1}{b} + \frac{1}{c} >1$ was investigated by Milanov-Tseng \cite{MT} and Rossi \cite{R} using Frobenius structures associated with the mirror potential $W$ for $\mathbb{P}^1_{a,b,c}$. For elliptic orbifold projective lines, global mirror symmetry and LG/CY correspondence were proved by Milanov-Ruan \cite{MR}, Krawitz-Shen  \cite{KS} and Milanov-Shen \cite{MSh}.

In mirror symmetry point of view, quantum cohomology rings of elliptic projective lines should be isomorphic to the Jacobian  rings of their mirror potentials :
\begin{equation}\label{eq:clMS_333}
\mathfrak{ks} : QH^\ast_{orb} (\mathbb{P}^1_{3,3,3}) \to Jac (W_{3,3,3}).
\end{equation}
We remark that $W_{a,b,c}$ has been explicitly computed in \cite{CHKL} and the computation in this paper should be helpful to understand such mirror symmetry isomorphism.


\subsection*{Acknowledgement}
We would like to express our gratitude to our advisor Cheol-Hyun Cho for suggesting the problem and also for valuable discussions. We also thanks Yefeng Shen for pointing out certain inaccuracies in the earlier version of the paper and explaining his work. The first author thanks Atsushi Takahashi for valuable discussions on the modularity of Gromov-Witten potentials, and thanks Jangwon Ju for explanation on theta series.

\section{Preliminaries}
In this section, we briefly review orbifold theory which will be used in the section \ref{classify}.

\subsection{Orbifolds}
The notion of orbifolds had been firstly introduced by Satake \cite{Sa} in the name of $V$-manifolds, and the theory of orbifolds was further developed by Thurston \cite{Thu79}. 
Here we follow the definition of orbifold as given in \cite{CR1}. (See \cite{ALR} for more details on other approaches.)

\begin{definition}\label{def:smmaporb} A paracompact Hausdorff space $X$ is called an $n$-dimensional orbifold if $X$ can be covered by open subsets $X= \cup U_i$ such that
\begin{itemize}
\item[(1)] there exists a homeomorphism $\phi_i : V_i / G_i \to U_i$ where $V_i$ is an open subset of $\R^n$ with an (effective) action of a finite group $G_i$;
\item[(2)] for a point $x \in U_i \cap U_j$, there is an open neighborhood $U_{ij}$ of $x$ in $X$ which is isomorphic to the quotient of an open domain $V_{ij}$ in $\R^n$ by a finite group $G_{ij}$, and embeddings $(V_{ij} \to V_i, G_{ij} \to G_i)$ and $(V_{ij} \to V_j, G_{ij} \to G_j)$ which are equivariant. 
\end{itemize}
The local model $(V_i, G_i, \phi_i)$ of $X$ is called a local uniformizing chart.
\end{definition}

In particular, for each $x \in X$, we can take an open neighborhood $U_x$ of $x$ in $X$ which is isomorphic $V_x / G_x$ (where $V_x \subset \R$ and $G_x$ is a finite group) such that the preimage of $x$ in $V_x$ is the single point which is fixed by $G_x$. We call $(V_x,G_x)$ a local uniformizing chart around $x$.
Now, for two orbifolds $X$ and $Y$, the morphism between them can be defined as follows.

\begin{definition}\label{def:smoothbetorb}
\begin{enumerate}
\item A smooth map $f$ between $X$ and $Y$ is a continuous map $f: X \to Y$, which has the following local property. For each $x \in X$, there exist uniformizing chart $(V_x, G_x)$ and $(V_{f(x)},G_{f(x)})$ of $x$ and $f(x)$ respectively and an injective group homomorphism $G_x \to G_{f(x)}$ such that $f$ admits a local smooth lifting $\WT{f}_{V_x V_{f(x)}} : V_x \to V_{f(x)}$ which is $(G_x, G_{f(x)})$-equivariant. 

\item A smooth map $f$ between $X$ and $Y$ is called good if it admits a collection of maps $\{ \WT{f}_{U U'}, \lambda \}$ which is called a compatible system of $f$ and defined as follows :

For each equivariant embedding $i : (V_2, G_2, \pi_2) \to (V_1, G_1, \pi_1)$ of local uniformizing charts of $X$, there is an equivariant embedding $\lambda(i) :  (V'_2, G'_2, \pi'_2) \to (V'_1, G'_1, \pi'_1)$ of charts of $Y$ with the following commuting diagram,
\begin{equation*}
\xymatrix{
{} && {} & V_0 \ar[dd]^{\WT{f}_{V_0 V'_0}} \\
V_2 \ar[rr]_{i} \ar[rrru]^{j \circ i} \ar[dd]_{\WT{f}_{V_2 V'_2}} && V_1 \ar[ru]_{j} \ar@<0.7ex>[dd]_{\WT{f}_{V_1 V'_1}} & {} \\
{} && {} & V'_0 \\
V'_2 \ar[rr]_{\lambda (i)} \ar[rrru]|!{[rrr];[uu]}\hole_{\lambda (j \circ i)} && V'_1 \ar[ru]_{\lambda (j)} & {}\\
}
\end{equation*}
where each maps are equivariant maps.
\end{enumerate}
\end{definition}

There is another notion of maps between orbifolds introduced by Takeuchi \cite{T}. We will compare two notions in Lemma \ref{lem:orbimap}.

\subsection{Orbifold fundamental group}\label{subsec:orbfundgp}

In this section, we recall the notion of the orbifold fundamental group (introduced by Thurston), which is closely related to orbifold covering theory explained below. This is analogous to the connection between usual covering theory and fundamental groups. It is enough to consider global quotient orbifolds for our purpose in this paper.
Consider a finite group action $G$ on a manifold $M$ which preserving orientation.
(One may consider an infinite group $G$, if an orbifold is represented as a quotient of non-compact manifold by an locally free action of a discrete group.)
Then generalized loops in $M$ are defined as follows:

\begin{definition} A path $\gamma : [0,1] \to M$ is called a generalized loop based at $\widetilde{x_0} \in M$ if $\gamma(0) = \widetilde{x_0}$, and there exists $g_{\gamma} \in G$ such that $\gamma(1)= g \cdot \gamma(0)$.
\end{definition}

Choose a point $\widetilde{x_0} \in M$ with $G_{\widetilde{x_0}}=1$, and let $\pi_1^{orb} ([M/G])$ be the set of equivalence classes of generalized loops based at $\widetilde{x_0}$ where the equivalence relation is given as homotopies fixing end points. One can check that $\pi_1^{orb} ([M/G])$ has a natural group structure by defining
$$[\gamma] \cdot [\delta] = [\gamma \# g_\gamma (\delta) ]$$
for generalized loops $\gamma, \delta$ based at $\widetilde{x_0}$ where `$\#$' denotes the concatenation of paths.

\begin{figure}
\begin{center}
\includegraphics[height=1.8in]{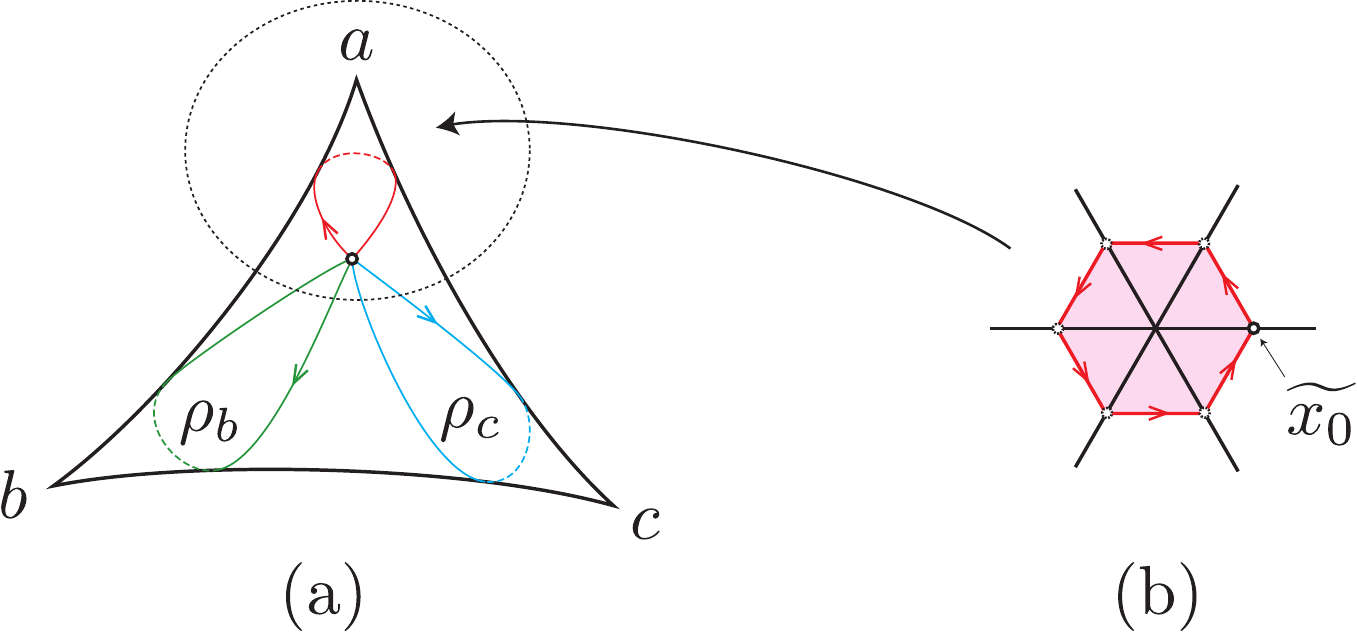}
\caption{(a) generators of $\pi_1^{orb} (\PP^1_{a,b,c})$ and (b) the relation $\rho_1^a=1$}\label{fig:orbfundgenpath}
\end{center}
\end{figure}

For example, $\pi_1^{orb} (\PP^1_{a,b,c})$ has a presentation
$$\pi_1^{orb} (\mathbb{P}_{a,b,c}) = <\rho_1, \rho_2, \rho_3 \,\, | \, \left(\rho_1 \right)^a = \left(\rho_2 \right)^b = \left(\rho_3 \right)^c = \rho_1 \rho_2 \rho_3  =1 >$$
where $\rho_1$, $\rho_2$ and $\rho_3$ are generalized loops in the universal cover of $\mathbb{P}^1_{a,b,c}$ whose images in the underlying space of the orbifold look as in (a) of Figure \ref{fig:orbfundgenpath}. The relation $\rho_1^a = 1$ can be seen in the uniformizing chart around the singular point of order $a$ (see (b) of Figure \ref{fig:orbfundgenpath}). One can observe the relation $\rho_1 \rho_2 \rho_3 =1$ even more directly on the orbifold itself.

\begin{remark}
See \cite{T} for details on the link between orbifold fundamental groups and orbifold coverings which we shall explain below.
\end{remark}

\subsection{Orbifold covering theory}\label{subsec:orbcovthe}

There is an analogue of covering space for orbifolds whose local model is $\R^n / G' \to \R^n / G$ for some finite group $G$ which acting on $\R^n$ with $G' \leq G$.
\begin{definition}\label{def:orbcover}\cite[Section 1]{T}
An orbifold $\WT{X}$ is called a covering orbifold, if there is a continuous surjective map $p : |\WT{X}| \to |{X}|$ satisfying following condition :

For each point $x \in |{X}|$, there is a local uniformizing chart $\WT{U}_x / G_x \cong U_x$ such that each point $\WT{x} \in p^{-1}(x)$ has a local uniformizing chart $\WT{U}_x / G_{x, i} \cong V_{x, i}$ for some $G_{x, i} \leq G_x$ which commutes following diagram
\begin{equation*}
	\xymatrix{
		{} 		& \WT{U}_x / G_{x,i} 	\ar[r]^{\cong}  \ar[d]^{q} 	& V_{x,i} \ar[d]^{p} \\
		\WT{U}_x \ar[r] \ar[ru]	& \WT{U}_x / G_x	\ar[r]^{\cong}	& U_x
	}
\end{equation*}
where $q$ is the natural projection.
\end{definition}

An orbifold which admits a manifold covering is called a {\em good orbifold}. For example, the orbifold projective lines $\PT$, $\mathbb{P}^1_{2,4,4}$ and $\mathbb{P}^1_{2,3,6}$ of our main concern are all good orbifolds, as they are given by quotients of a 2-torus.
Throughout the section, we assume that all orbifolds are good.

%

Following \cite{T} we introduce the notion of orbi-map.
\begin{definition}\label{def:orbimaps}\cite[Section 2]{T} An orbi-map $f : X \to Y$ consists of a continuous map $h : |X| \to |Y|$ between underlying spaces and a fixed continuous map $\tilde{h} : \WT{X} \to \WT{Y}$ which satisfy
\begin{enumerate}
\item $h \circ p = q \circ \tilde{h}$
\item For each $\sigma \in {\rm Aut} ( \WT{X}, p) (\cong \pi_1^{orb} (X))$, there exists $\tau \in {\rm Aut} (\WT{Y} ,q)$ such that $\tilde{h} \circ \sigma = \tau \circ \tilde{h}$
\item $h (|X|) $ does not lie in the singular loci of $Y$ entirely.
\end{enumerate} 
\end{definition}

\begin{remark}
Indeed, covering theory in \cite{T} only concerns about good orbifolds.
\end{remark}

For orbi-maps, we have usual lifting theorems in covering theory as well, whose proof is not very much different from the standard one.

\begin{prop}\cite[Proposition 2.7]{T} \label{prop:cov_lifting}
Let $f : ({X}, x) \to ({Y}, y)$ be an orbi-map and $p : ({Y}', y') \to ({Y}, y)$ be a covering. Then $f$ can be lifted to an orbi-map $\WT{f} : {X} \to {Y}'$ if and only if $f_* \pi^{orb}_1 (X, x) \subset p_* \pi^{orb}_1 ({Y}', y')$.
\end{prop}

As we will often use lifting theorems for orbifold coverings by Takeuchi \cite{T}, we also introduce another notion of morphisms between orbifolds which Takeuchi called orbi-maps. However, we will show in Subsection \ref{subsec:orbimap} that two notions are equivalent for our main examples.

\subsection{Orbi-maps between two dimensional orbifolds}\label{subsec:orbimap}

While orbifold covering theory is well-established for orbi-maps by Takeuchi \cite{T}, orbifold quantum cohomology is defined by counting good maps given in (2) of Definition \ref{def:smoothbetorb}. 
In the case of elliptic $\PP^1$ orbifolds, it turns out that maps given in (1) of Definition \ref{def:smoothbetorb} satisfy axioms in Definition \ref{def:orbimaps}.

\begin{remark}
From \cite[Lemma 4.4.11]{CR2}, a smooth map $f : X \to X'$ between two orbifolds $X$ and $X'$ is a good map with a unique compatible system up to isomorphism, if the inverse image of the regular part of $X'$ is an open dense and conneted subset of $X$. Note that a non-constant smooth map contributing to the quantum product for $\mathbb{P}^1_{a,b,c}$ automatically satisfies this property.
%
%
\end{remark}

Consider two (good) orbifolds $X$ and $Y$ which admit manifold universal covering spaces $p:\WT{X} \to X$ and $q: \WT{Y} \to Y$, respectively. Assume that the deck transformation action of $p$ and $q$ are orientaion preserving. Moreover, assume that both $X$ and $Y$ are two dimensional (which is the case of our main interest).
Note that singular loci of $X$ and $Y$ are sets of isolated points.

\begin{lemma}\label{lem:orbimap}
With the setting as above, any non-constant smooth map $\phi : X \to Y$ satisfies the axioms in Definition \ref{def:orbimaps} if $\dim X = \dim Y =2$.
\end{lemma}

\begin{proof}
We first show that there is a continuous map $\tilde{\phi} : \WT{X} \to \WT{Y}$ which lifts $\phi : X \to Y$.
\begin{equation}\label{eq:diaorbigood}
\xymatrix{\WT{X} \ar[rrd]^{\phi \circ p} \ar@{-->}[rr]^{\tilde{\phi}} \ar[d]_p && \WT{Y} \ar[d]^q \\
X \ar[rr]_\phi  && Y}
\end{equation}
What we want to have is basically a lift of the map $\phi \circ p : \WT{X} \to Y$. We claim that at each point $\tilde{x} \in \tilde{X}$ there is a local lifting of $\phi \circ p$. Let $x:=p (\tilde{x})$.
Then one can find a neighborhood of $\tilde{x}$ which uniformizes $X$ locally around $x$. Since the same is true for any point in the inverse image $q^{-1} (\phi(x))$, we can find a local lifting of $\phi$ around $\tilde{x}$ by the properties of orbifold maps.

By gathering such a neighborhood for each $\tilde{x} \in \WT{X}$, we obtain a open covering $\WT{\mathcal{U}} = \left\{ \WT{U}_i \, : \, \WT{U} \subset \WT{X}, i \in I \right\}$ of $\WT{X}$ which consists of open subsets of $\WT{X}$ on which $\phi$ can be locally lifted. For each $\WT{U}_i \in \WT{\mathcal{U}}$, we fix a local lifting $\tilde{\phi}_i$ of $\phi$. On the intersection of two open subsets $\WT{U}_i$ and $\WT{U}_j$ in $\WT{\mathcal{U}}$, two local liftings $\tilde{\phi}_i$ and $\tilde{\phi}_j$ differ by an element $g_{ij}$ of ${\rm Aut} (\WT{Y}, q) \cong \pi_1^{orb} (Y)$. i.e. 
\begin{equation}\label{eq:differgij}
\tilde{\phi}_i |_{\WT{U}_i \cap \WT{U}_j} = g_{ij} \circ \tilde{\phi}_j |_{\WT{U}_i \cap \WT{U}_j}.
\end{equation}

Note that $\{g_{ij} \, : \, i,j \in I \}$ satisfies the usual cocycle condition, that is,
\begin{equation}\label{eq:gijcocycle}
g_{ij} g_{jk} g_{ki} =1.
\end{equation}
\eqref{eq:gijcocycle} follows from
\begin{eqnarray*}
\tilde{\phi}_i &=& g_{ij} \circ \tilde{\phi}_j \\
&=& (g_{ij} g_{jk} )\circ \tilde{\phi}_k \\
&=& (g_{ij} g_{jk} g_{ki} ) \circ \tilde{\phi}_i
\end{eqnarray*}
on $\WT{U}_i \cap \WT{U}_j \cap \WT{U}_k$ and the fact that the action of $\pi_1^{orb} (Y)$ on $Y$ is free generically. (Recall that $\phi$ is a non-constant map.)

Therefore, $\{g_{ij} \}_{i,j \in I}$ defines a principal $\pi_1^{orb} (Y)$-bundle over $\WT{X}$  or equivalently a covering space of $\WT{X}$. Here, $g_{ij}$ glues $\WT{U}_i \times \pi_1^{orb} (Y)$ and $\WT{U}_j \times \pi_1^{orb} (Y)$ by the left multiplication so that the resulting bundle admits the right action of $\pi_1^{orb} (Y)$. Since $\pi_1^{orb} (Y)$ is discrete and $\WT{X}$ is simply connected, this bundle should be trivial. Therefore, the cocycle $\{g_{ij}\}$ is also trivial up to coboundary. i.e. there exists a collection $\{(\epsilon_i,\WT{U}_i) \in \pi_1^{orb} (Y) \times \WT{\mathcal{U}} \, : \, i \in I \}$ of elements of  $\pi_1^{orb} (Y)$ each of which is associated with an open subset in $\WT{\mathcal{U}}$ such that
\begin{equation}\label{eq:gijcobdy}
g_{ij} = \epsilon_i^{-1} \epsilon_j.
\end{equation}
(In other words, $\{ \sigma_i \}$ trivializes the principal bundle corresponding to the data $\{g_{ij}\}$.)

From \eqref{eq:differgij} and \eqref{eq:gijcobdy}, we have
$$\epsilon_i \tilde{\phi}_i = \epsilon_j \tilde{\phi}_j$$
on $\WT{U}_i \cap \WT{U}_j$. If we set $\tilde{\phi}'_i := \epsilon_i \tilde{\phi}_i$, then $\{\tilde{\phi}'_i \}_{i \in I}$ gives a collection of local liftings of $\phi$ any two of which agree on their common domain. Denote the resulting global lifting of $\phi$ by $\tilde{h}$.

We next check the second axiom of Definition \ref{def:orbimaps}. Let $\sigma$ be a deck transformation of the covering $p: \WT{X} \to X$, that is, an element of ${\rm Aut} \, (\WT{X}, p)$.  Then $\tilde{h} \circ \sigma$ is a lifting of $\phi \circ p$ because 
\begin{eqnarray*}
q \circ \tilde{h} \circ \sigma &=& (\phi \circ p) \circ \sigma\\
&=& \phi \circ (p \circ \sigma) \\
&=& \phi \circ p.
\end{eqnarray*}

Since both $\tilde{h}$ and $\tilde{h} \circ \sigma$ are liftings of $\phi \circ p$, one can find an element $\tau_{\tilde{x}}$ in ${\rm Aut}\,(Y,q)$ for each $\tilde{x} \in \WT{X}$ such that
$$\tilde{h} \circ \sigma (\tilde{x}) = \tau_{\tilde{x}} ( \tilde{h} ( \tilde{x})).$$
Since $\WT{X}$ is connected and $\pi_1^{orb} (Y)$ is discrete, $\tau_{\tilde{x}}$ has to be independent of $\tilde{x}$. This gives an element $\tau$ in the second property of orbi-maps.
\begin{equation}
\xymatrix{\WT{X} \ar[rr]^{\tilde{h} \circ \sigma, \,\, \tilde{h}} \ar[drr]_{\phi \circ p} && \WT{Y} \ar[d]^q  \\
&& Y
}
\end{equation}

Finally, the third condition of orbi-maps follows obviously since we are only considering non-constant morphisms between orbifolds with the same dimension.

%
%
\end{proof}

\begin{remark}
From the proof, we see that the lemma also holds for a smooth map between two good orbifolds of general dimensions which does not send a whole open subset to a fixed locus.
\end{remark}

\section{Gromov-Witten theory of orbifolds}\label{sec:GWtheoryorb}
In this section, we briefly review the quantum cohomology of orbifolds developed by Chen and Ruan. The key ingredients in defining the product on this cohomology are holomorphic orbi-spheres (or orbifold stable maps in general) in orbifolds with three marked points. If we consider such spheres with arbitrary number of markings as well, then we obtain the orbifold (genus-0) Gromov-Witten invariants \cite{CR2} (See Subsection \ref{subsec:desmod} and \ref{subsec:GWorb}, also.)

At the end of the section, we will come back to our main examples, orbifold projective lines $\mathbb{P}^1_{a,b,c}$ and their Gromov-Witten theory. We are particularly interested in $\PP^1_{a,b,c}$ with $\frac{1}{a} + \frac{1}{b} + \frac{1}{c} = 1$ which are in fact quotients of an elliptic curve, which exhibit a lots of number theoritic phenomenons. We remark that Satake and Takahashi \cite{ST} provided the full genus-0 potential of $\PP^1_{3,3,3}$ using the algebraic method. We will also briefly recall their work.


\subsection{Description of $\OL{\CM}_{g,k,\beta}(X)$}\label{subsec:desmod}
Let $(X,\omega)$ be an compact effective symplectic orbifold with a compatible almost complex structure $J$. (See \cite[Definition 2.1.1, 2.1.5]{CR2}.) We begin with the description of the compactified moduli space of orbifold stable maps into $X$. Details can be found in \cite{CR2}.

\begin{definition}[\cite{CR2}, Definition 2.2.2]
An orbi-Riemann surface of genus $g$ is a triple $(\Sigma, \bz, \bm)$ :
\begin{itemize}
\item $\Sigma$ is a smooth Riemann surface of genus $g$.
\item $\bz = ( z_1, \cdots, z_k)$ is a set of orbi-singular points on $\Sigma$ with isotropy group of order $\bm = ( m_1, \cdots, m_k )$ for some integer $m_i (\geq 2)$.
The orbifold structure on $\Sigma$ is given as follows: at each point $z_i$, a disc neighborhood of $z_i$ is uniformized by the branched covering map $z \to z^{m_i}$.
\end{itemize}
\end{definition}

In order to compactify the moduli space, we should also include nodal Riemann surfaces as domains of holomorphic maps.

\begin{definition}[\cite{CR2}, Definition 2.3.1]
A nodal Riemann surface with $k$ marked points is a pair $(\Sigma, \bz)$ of a connected topological space $\Sigma =  \bigcup \pi_{\Sigma_{\nu}}(\Sigma_{\nu})$ and a set of $k$-distinct points $\bz = (z_1, \cdots, z_k)$ in $\Sigma$ with the following properties.
	\begin{itemize}
		\item $\Sigma_{\nu}$ is a smooth Riemann surface of genus $g_{\nu}$, and $\pi_{\nu} : \Sigma_{\nu} \to \Sigma$ is a continuous map. The number of $\Sigma_{\nu}$ is finite.
		\item For each $z \in \Sigma_{\nu}$, there is a neighborhood of it such that the restriction of $\pi_{\nu} : \Sigma_{\nu} \to \Sigma$ to this neighborhood is a homeomorphism to its image.
		\item For each $z \in \Sigma$, we have $\sum_{\nu} \#\pi_{\nu}^{-1} (z) \leq 2$, and the set of nodal points $\left\{ z | \sum_{\nu} \#\pi_{\nu}^{-1} (z) = 2\right\}$ is finite.
		\item $\bz$ does not contain any nodal point.
	\end{itemize}
\end{definition}

We next allow cone singularities on nodal Riemann surfaces.

\begin{definition}[\cite{CR2}, Definition 2.3.2]\label{def:nodalorb}
A nodal orbi-Riemann surface is a nodal marked Riemann surface $(\Sigma, \bz)$ with an orbifold structure as follows:
	\begin{itemize}
		\item The set of orbi-singular points
		is contained in the set of marked points and nodal points $\bz$;
		\item A disk neighborhood of a marked point is uniformized by a branched covering map $z \to z^{m_i}$;
		\item A neighborhood of a nodal point is uniformized by the chart $( \WT{U}, \Z_{n_j})$, where $\WT{U} = \left\{ (x, y) \in \C^2 \big{|} xy = 0 \right\}$ on which $\Z_{n_j}$ acts by $e^{\frac{2 \pi i}{n_j}} \cdot (x, y) = (e^{\frac{2 \pi i}{n_j}} x, e^{-\frac{2 \pi i}{n_j}}y)$.
	\end{itemize}
Here $m_i$ and $n_j$ are allowed to be $1$. We denote the corresponding nodal orbi-Riemann surface by $(\Sigma, \bz, \bm, \bn)$, and if there is no confusion, simply write it as $(\Sigma, \bz)$. (See Figure \ref{fig:orb_stab_map}.)
\end{definition}
\begin{figure}
\begin{center}
\includegraphics[height=1.3in]{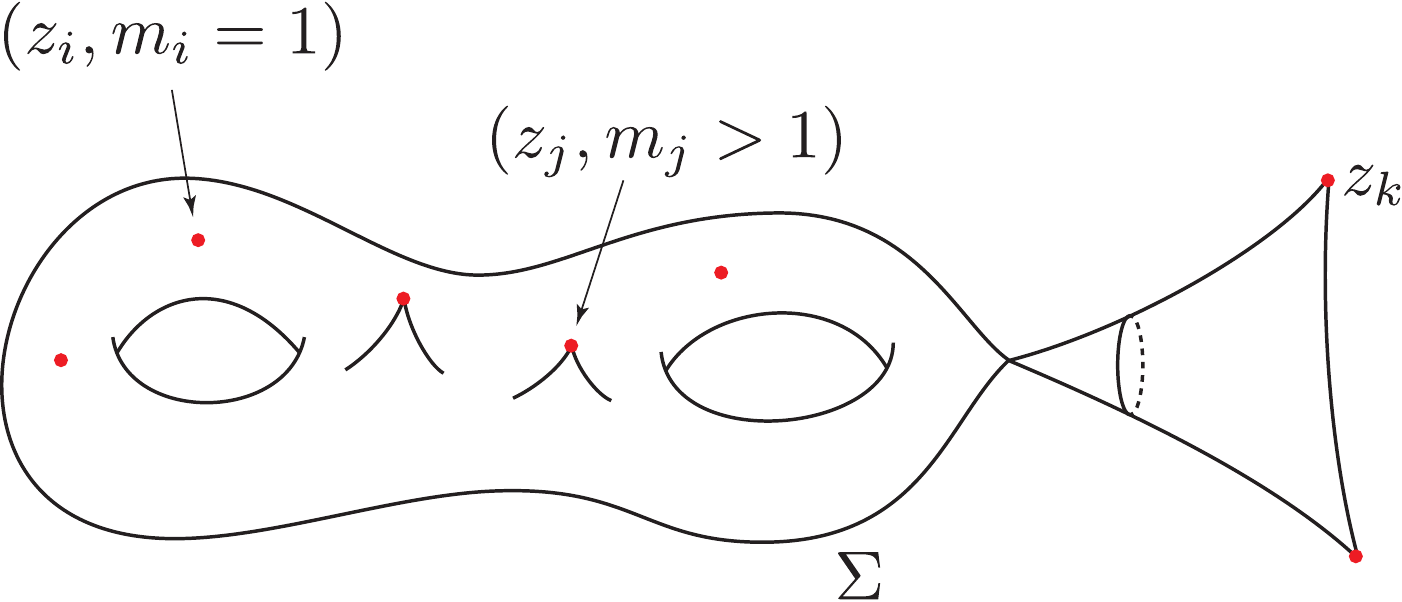}
\caption{A nodal orbi-Riemann surface}\label{fig:orb_stab_map}
\end{center}
\end{figure}
Having a nodal orbi-Riemann surface as the domain, an orbifold stable map is defined as follows.

\begin{definition}[\cite{CR2}, Definition 2.3.3]\label{def:crobstab}
For given an almost complex orbifold $(X, J)$, an orbifold stable map is a triple $(f, (\Sigma, \bz, \bm, \bn), \xi)$ :
	\begin{itemize}
		\item $f : \Sigma \to |X|$ is a continuous map from a nodal Riemann surface $\Sigma$ such that $f_{\nu} = f \circ \pi_{\nu}$ is a $J$-holomorphic map.
		\item Consider the local lifting $\WT{f}_{V_x V_{f(x)}} : (V_x, G_x) \to (V_{f(x)}, G_{f(x)})$ of $f$. Then the homomorphism $G_x \to G_{f(x)}$ is injective.
		\item Let $k_{\nu}$ be the number of points in $\Sigma_{\nu}$ which is marked or nodal. If $f_{\nu}$ is a constant map, then $k_{\nu} + 2 g_{\nu} \geq 3$.
		\item $\xi$ is an isomorphism class of compatible systems.
		\end{itemize}
\end{definition}
For the definition of an isomorphism between compatible systems, see \cite[Definition 4.4.4]{CR2}.




%

We are now ready to define the moduli space relevant to the orbifold Gromov-Witten invariants of ${X}$.

\begin{definition}\label{def:GWgauge_equiv}
\begin{enumerate}
\item
Two stable maps $\big(f, (\Sigma, \bz), \xi \big)$ and $\big(f', (\Sigma', \bz'), \xi' \big)$ are equivalent if there is an isomorphism $\theta : (\Sigma, \bz, \bm, \bn) \to (\Sigma', \bz', \bm', \bn')$ such that $f' \circ \theta = f$ and $\xi' \circ \theta = \xi$.
\item
Given a homology class $\beta \in H_2(|{X}|; \Z)$, $\OL{\CM}_{g,n,\beta}({X}, J)$ is defined as the moduli space of equivalence classes of orbifold stable maps of genus $g$, with $k$ marked points, and of homology class $\beta$.
\end{enumerate}
\end{definition}

%
%
%


%

\subsection{Gromov-Witten invariants of orbifolds}\label{subsec:GWorb}

%
%

Now, we recall the definition of the orbifold cohomology group $H^\ast_{orb} ({X}, \Q)$ from \cite{CR1} which is freely generated by the elements of the cohomology groups of twisted sectors of ${X}$ (i.e., $H^\ast_{orb} ({X},\Q) = H^\ast (IX, \Q)$ as a vector space, where $IX$ is the inertia orbifold of ${X}$.) Here, the degrees of elements in $H^\ast ( {X}_{(g)},Q)$ are shifted by $2 \iota(g)$ where ${X}_{(g)}$ is the twisted sector associated with the conjugacy class $(g)$ and $\iota(g)$ is the age of an element $g$ in a local group. 
$H^*_{orb} ({X}, \Q)$ also admits a natural Poincar\`e pairing which is compatible with these shifted degrees:
\begin{equation}\label{eq:P_pair}
\int^{orb}_{{X}} (-) \cup_{orb} (-) : H^\ast_{orb} ({X}, \Q) \otimes H^\ast_{orb} ({X}, \Q) \to \Q.
\end{equation}

 We fix a $\Q$-basis $\{\gamma_i \}_{i=1,\cdots,N}$ of $H^\ast_{orb} ({X}, \Q)$. 
 Then the $k$-fold Gromov-Witten invariants is defined by the following equation:
\begin{equation}\label{eq:nGW}
\langle \gamma_1 , \cdots, \gamma_k \rangle^{{X}}_{g, k, \beta} := \int_{[\OL{\CM}_{0,k,\beta}({X})]^{vir}} ev_1^* \gamma_1 \wedge \cdots \wedge ev_k^* \gamma_k.
\end{equation}
We also define $\langle \gamma_1 , \cdots, \gamma_k \rangle^{{X}}_{g, k}$ to be the weighted sum $\sum_{\beta} \langle \gamma_1 , \cdots, \gamma_k \rangle^{{X}}_{g, k, \beta} q^{\omega (\beta)}$.

\begin{remark}
The compactified moduli space $\OL{\CM}_{0,k,\beta}({X})$ admits a virtual fundamental class $[\OL{\CM}_{0,k,\beta}({X})]^{vir}$ which can be defined with help of an abstract perturbation technique in general. (Readers are referred  to \cite{CR2} for more details.)
\end{remark}

For a tuple $\boldsymbol{x} = ({X}_{(g_1)}, {X}_{(g_2)}, \cdots, {X}_{(g_k)})$ of twisted sectors, we say that $\big(f, (\Sigma, \bz), \xi \big)$ is of type $\boldsymbol{x}$, if orbi-insertions at the marked point $z_i$ lies in $H^\ast ({X}_{(g_i)}, \Q)$ for each $i$. Let $\OL{\CM}_{g,k,\beta}({X}, J, \beta, \boldsymbol{x})$ denote the moduli space of equivalence classes of orbifold stable maps of genus $g$ with $k$ marked points and of homology class $\beta$ and type $\boldsymbol{x}$. Then $\OL{\CM}_{0,k,\beta}({X})$ is the union of $\OL{\CM}_{g,k,\beta}({X}, J, \beta, \boldsymbol{x})$ over all types $\boldsymbol{x}$, and the integration in \eqref{eq:nGW} is nonzero on components $\OL{\CM}_{g,k,\beta}({X}, J, \beta, \boldsymbol{x})$ with $\gamma_i \in H^\ast ({X}_{(g_i)}, \Q)$.

For later purpose, we give the virtual dimension of $\OL{\CM}_{g,k,\beta}({X}, J, \beta, \boldsymbol{x})$ explicitly:
\begin{equation}\label{eq:virdimMod}
2c_1(T{X}) (\beta) + 2(n-3)(1-g) + 2k -2 \iota(\boldsymbol{x}),
\end{equation}
where $2n = \rm{dim}_{\R} X$ and $\iota(\boldsymbol{x}) = \sum_{i=1}^{k} \iota(g_i)$. (See \cite[Proposition 3.2.5]{CR2}.)

\begin{remark}\label{rmk:c1vanvirdim}
If $c_1$ of ${X}$ vanishes as in our elliptic examples, the virtual dimension of the moduli is independent of the homology class $\beta$.
In particular, $n = 1$ and $g = 0$ in our main examples.
\end{remark}

If we set $\bt := \sum t_i \gamma_i$, then the generating function for the Gromov-Witten invariants is defined as
\begin{equation*}
F^{{X}}_0 (\bt):= \sum_{k, \beta} \frac{1}{k!} \langle \bt, \cdots, \bt \rangle^{{X}}_{0, k, \beta} q^{\omega(\beta)},
\end{equation*}
which we will call the genus-0 Gromov-Witten potential for ${X}$.

In particular when $k=3$, the counting given in \eqref{eq:nGW} defines a product $\ast$ on $H^*_{orb} ({X}, \Q)$ which is called the quantum product. More precisely,
$$ \int_{{X}}^{orb} (\gamma_i \ast \gamma_j) \cup_{orb} \gamma_l := \langle  \gamma_i, \gamma_j, \gamma_l \rangle_{0,3}^{{X}},$$
or equivalently,
$$\gamma_i \ast \gamma_j := \sum_{l=1}^N \sum_{\beta} \langle  \gamma_i, \gamma_j, \gamma_l \rangle^{{X}}_{0,3,\beta} PD(\gamma_l) q^{\omega (\beta)}$$
where $PD(-)$ denotes the Poincar\`e dual with respect to the pairing \eqref{eq:P_pair}. Therefore, \eqref{eq:nGW} with $n=3$ gives structure constants of this product. The associativity of $\ast$ is proved in \cite{CR2}. We remark that what we have defined is the small quantum cohomology of ${X}$ while the big quantum cohomology involves the full Gromov-Witten invariants.



\subsection{Elliptic orbifolds $\PP^1_{a,b,c}$ and review on Satake-Takahashi's work}\label{subsec:GWforP}

We now focus on elliptic orbifolds with three cone points $\PP^1_{a,b,c}$ and its Gromov-Witten potential. $\PP^1_{a,b,c}$ is elliptic if and only if $\frac{1}{a} + \frac{1}{b} + \frac{1}{c} =1$, and hence there are three elliptic orbifold projective lines $\PP^1_{a,b,c}$ where $(a,b,c)$ are $(3,3,3),(2,3,6)$ and $(2,4,4)$. They are called elliptic since these orbifolds can be described as a global quotient of an elliptic curve $E$ by a finite cyclic group $G$.

We first fix the notation for generators of their orbifold cohomology rings in the following way:
Let $w_1$, $w_2$, and $w_3$ be the three cone points $\PP^1_{a,b,c}$ with isotropy groups $\Z_a$, $\Z_b$, and $\Z_c$, respectively. We fix a choice of a $\Q$-basis of $H^*_{orb}(\PP^1_{a,b,c}, \Q)$, which is fairly standard.
The $\Q$-basis 
\begin{equation}\label{eq:fixbasisnote}
\left\{ 1, \da{1}{a}, \cdots, \da{(a-1)}{a}, \db{1}{b}, \cdots, \db{(b-1)}{b}, \dc{1}{c}, \cdots, \dc{(c-1)}{c}, \pt \right\}
\end{equation}
of $H^*_{orb}(\PP^1_{a,b,c}, \Q)$ is defined by the following conditions.
The basis of smooth sector are
\begin{align*}
&H^0_{orb}(\PP^1_{a,b,c}, \Q) = \Q \langle 1 \rangle, \quad \quad \quad H^2_{orb}(\PP^1_{a,b,c}, \Q) = \Q \langle \pt \rangle.
\end{align*}
For twist sectors, let $\da{j}{a} \in H^{\frac{2j}{a}}_{orb}(\PP^1_{a,b,c}, \Q)$, $\db{j}{b} \in H^{\frac{2j}{b}}_{orb}(\PP^1_{a,b,c}, \Q)$, and $\dc{j}{c} \in H^{\frac{2j}{c}}_{orb}(\PP^1_{a,b,c}, \Q)$ which are supported at singular points $w_1$, $w_2$, and $w_3$, respectively.
For $\bt := \sum t_{j,i} \Delta^i_j$,

Orbifold cup products with respect to these basis are given as follows.
$$\int_{X}^{orb} \Delta_{1}^{j/a} \cup \Delta_{1}^{k/a} = \frac{1}{a} \delta_{j+k-a , 0} \quad \int_{X}^{orb} \Delta_{1}^{j/b} \cup \Delta_{1}^{k/b} = \frac{1}{b} \delta_{j+k-b,0},$$
$$\int_{X}^{orb} \Delta_{1}^{j/c} \cup \Delta_{1}^{k/c} = \frac{1}{a} \delta_{j+k-c,0} \quad \int_{X}^{orb} 1 \cup [pt] = 1$$
where $\delta_{i,j}$ is $1$ if $i=j$ and zero otherwise. The last cup product does not have any fraction since both $1$ and $[pt]$ live in smooth(untwisted) section of $\mathcal{I}{X}$.

\begin{remark}\label{rmk:warnPD}
Readers are hereby warned that the Poincar\`e dual $PD(\Delta_1^{j/a})$ of $\Delta_1^{j/a}$ is not $\Delta_1^{(a-j)/a}$, but $a \times \Delta_1^{(a-j)/a}$, and the same happens for $b$ and $c$. However, $1$ and $[pt]$ are still Poincar\`e dual to each other.
\end{remark}

In the remaining part, we briefly recall the work of Satake and Takahashi \cite{ST} on $\PT$. We first give a description of $\PT$ as a quotient of an elliptic curve. Let $E$ be the elliptic curve associated with the lattice $ \Z  \langle 1, \tau \rangle$ in $\C$ where $\tau =\exp \left(\frac{2 \pi \sqrt{-1}}{3} \right)$. Then the $\Z_3$-action on $\C$ generated by $2 \pi /3$-rotation descends to $E$ since this action preserves the lattice $ \Z  \langle 1, \tau \rangle$. By taking quotients of $E$ via the induced $\Z_3$-action, we finally obtain the global quotient orbifold $\PT = [E / \Z_3]$. (The shaded region in (a) of Figure \ref{fig:P1333} represents a fundamental domain of the $\Z_3$-action on $E$.) Since each fixed point in $E$ has the isotropy group isomorphic to $\Z_3$, $\PT$ has three cone points each of which has $\Z_3$-singularity. We denote these singular points by $w_1$, $w_2$ and $w_3$, respectively. Therefore, the inertia orbifold $\mathcal{I}\PT$ consists of the trivial sector together with three $B\Z_3$(equivalent to $[\Z_3 \ltimes \{\ast \}]$)'s which are associated with the point $w_i$'s.


\begin{figure}
\begin{center}
\includegraphics[height=2.3in]{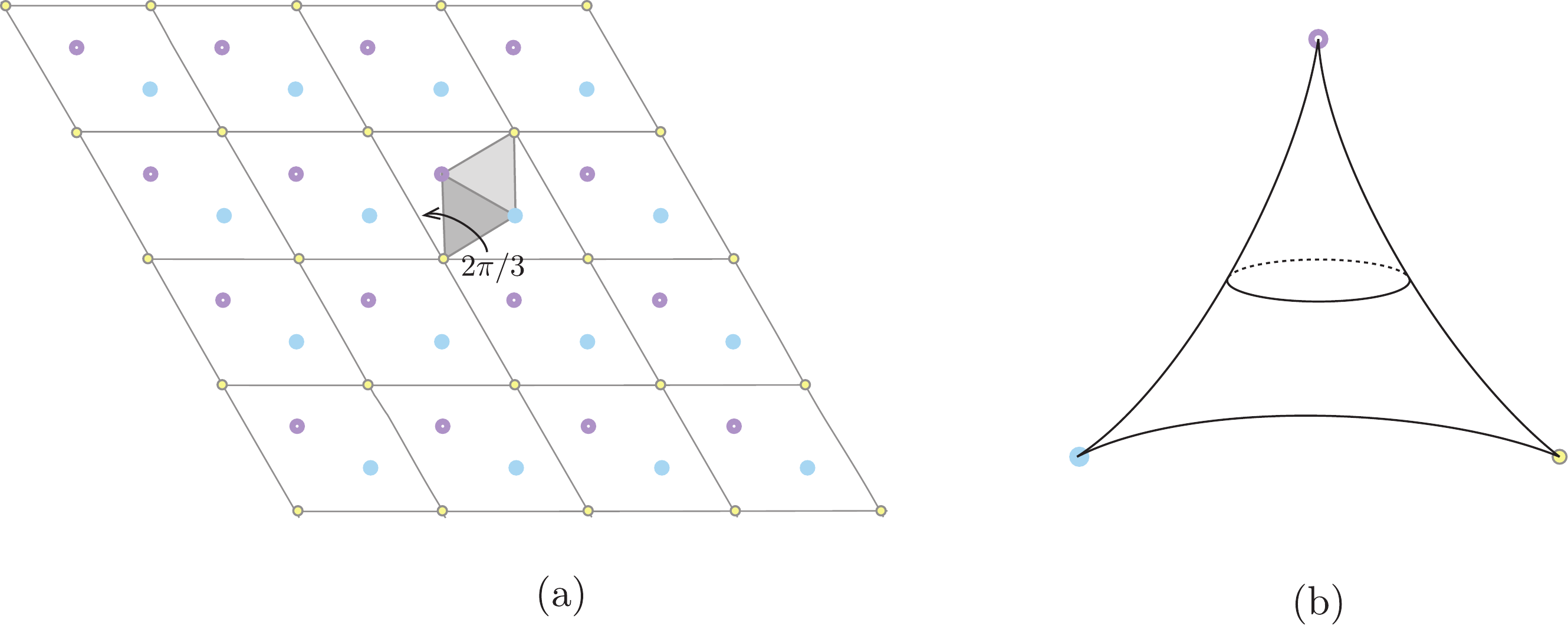}
\caption{(a) The $\Z_3$-action on $E$ and (b) its quotient}\label{fig:P1333}
\end{center}
\end{figure}

\begin{remark}\label{rmk:orbcoverE}
Consider the universal covering $\C \to E$ of the elliptic curve.
The composition $p: \C \to E \to \PT$ as well as the quotient map $E \to \PT$ is a holomorphic orbifold covering map in the sense of \cite{T} (see Definition \ref{def:orbcover}, also). Indeed, $p$ is the orbifold universal cover. We will use this fact crucially to classify holomorphic orbi-spheres in $\PT$.
\end{remark}
%

Following the notation in \eqref{eq:fixbasisnote}, the $\Q$-basis of $H^*_{orb}(\PT, \Q)$ is given by
\begin{align*}
&H^0_{orb}(\PT, \Q) = \Q \langle 1 \rangle, \quad \quad \quad && H^2_{orb}(\PT, \Q) = \Q \langle \pt \rangle,\\
&H^\frac{2}{3}_{orb}(\PT, \Q) = \Q \langle \da{1}{3}, \db{1}{3}, \dc{1}{3} \rangle, \quad && H^\frac{4}{3}_{orb}(\PT, \Q) = \Q \langle \da{2}{3}, \db{2}{3}, \dc{2}{3} \rangle,
\end{align*}
and the Poincar\`e pairing for $\Delta^{i}_{j}$'s is determined by
\begin{equation*}
\int_{\PT}^{orb} \Delta^i_j \cup_{orb} \Delta^k_l = \left\{
\begin{array}{cc}
\frac{1}{3}	, & \text{if $i + k = 1$ and $j = l$},\\
0, & \text{otherwise.}
\end{array}\right.
\end{equation*}


Satake and Takahashi \cite[Theorem 3.1]{ST} calculated the genus 0-Gromov-Witten potential of $\PT$ and the quantum product term can be written as
\begin{equation}\label{eq:f0f1}
\begin{array}{rl}
f_0 (q) &:= \sum_{d \geq 0} \langle \da{1}{3}, \db{1}{3}, \dc{1}{3} \rangle_{0,3,d} q^d = \frac{\eta(q^9)^3}{\eta(q^3)},	\\
f_1 (q) &:= \sum_{d \geq 0} \langle \da{1}{3}, \da{1}{3}, \da{1}{3} \rangle_{0,3,d} q^d =\left( 1 + \frac{1}{3} \left(\frac{\eta(q)}{\eta(q^9)}\right)^3 \right) f_0(q).
\end{array}
\end{equation}
Here, $\eta(q)$ is the Dedekind's eta function
\begin{equation*} 
	\eta(\tau) := \exp \left(\frac{\pi \sqrt{-1} \tau}{2}\right) \prod_{n = 1}^{\infty} \left(1 - q^n \right), \quad
	q = \exp(2 \pi \sqrt{-1} \tau)
\end{equation*}
for $\tau \in \mathbb{H} := \big\{ \tau \in \C \,\,|\,\, \text{Im} (\tau) > 0 \big\}$.

Write $f_0 (q) = \sum_{N \geq 1} a_N q^N$.
The Fourier coefficients $a_N$ of $f_0$ depends on the prime factorization of $N$ (or more precisely the quadratic reciprocity of $N$), and is given by
\begin{equation*}
a_N= \left\{
\begin{array}{cl}
0 & n  >0, \,\,\,\mbox{or one of}\,\,\, {m_j} \,\,\, \mbox{is odd}\\
(n_1 +1)  \cdots (n_k +1) & \mbox{otherwise}
\end{array}
\right.
\end{equation*}
for $N = 3^n p_1^{n_1} \cdots p_k^{n_k} q_1^{m_1} \cdots q_l^{m_l}$ where $p_i$ is a prime number with $p_i \equiv 1 (\mod 3)$, and $q_i$ is a prime number with $q_i \equiv 2 \,\, (\!\!\!\!\mod 3)$. (See \cite{S}.) The Fourier coefficients of $f_1$ also has a similar description which we will give in Appendix.

We also provide first few terms of $f_0$ and $f_1$ for readers to get an impression:
$$f_0 = q + q^4 + 2 q^7 + 2 q^{13} + q^{16} + 2 q^{19} + O(q^{24}),$$
$$f_1 = \frac{1}{3} + 2 q^3 + 2 q^9 + 2 q^{12} + 4 q^{21} + O(q^{24}).$$

\section{Holomorphic orbifold maps}\label{classify}
As mentioned in the introduction, our main goal is to compute the (quantum) product structure of $H^\ast_{orb} (\mathbb{P}^1_{a,b,c}, \Q)$ where $(a,b,c)$ is one of $(3,3,3)$, $(2,3,6)$, and $(2,4,4)$.
Throughout the section, $\PP^1_{a,b,c}$ denotes one of elliptic orbifolds $\PT$, $\PH$, and $\PQ$. 
In order to do this, we have to count holomorphic orbi-spheres in (or stable maps into) $\mathbb{P}^1_{a,b,c}$ with three markings. In this section, we first characterize these holomorphic maps and find their properties which are useful to classify holomorphic orbi-spheres in $\PP^1_{a,b,c}$.
We will see that if $\big(f, (\PP^1, \bz), \xi \big)$ is a non-constant orbifold stable map into elliptic $\mathbb{P}^1_{a,b,c}$ of type $\boldsymbol{x}$ , $\bz$ can not contain a smooth point. Thus, we may assume that $\boldsymbol{x}$ is a triple of twisted sectors. (See the discussion after the proof of Lemma \ref{lem:nodal} below.) 

Recall that the type $\boldsymbol{x}$ determines the virtual dimension of a component of the moduli of orbifold stable maps containing $\big(f, (\PP^1, \bz), \xi \big)$ as well as the orbifold structure of domain orbi-Riemann sphere $(\PP^1, \bz)$ (Remark \ref{rmk:c1vanvirdim}). In fact, the virtual dimension is given as
$$\text{vir.dim} \,\OL{\CM}_{0, 3}(\PP^1_{a,b,c}, J, \beta, \boldsymbol{x}) = 2  - 2 \iota(\boldsymbol{x}).$$
As we only consider the $0$-dimensional moduli for the quantum product, this gives a restriction on the type, that is, $\iota(\boldsymbol{x})=1$. If we impose an additional condition on the degrees of inputs for holomorphic orbi-sphere of type $\boldsymbol{x}$ with $\iota(\boldsymbol{x})=1$, we can show that it is actually an orbifold covering. This will be shown in Section \ref{subsec:orbsphorbcover}.

As the first step, we show that there is no contribution to the quantum product from a degenerate orbi-sphere, which is an element lying on the boundary of $\OL{\CM}_{0,3,\beta}(\PP^1_{a,b,c})$.

\subsection{Considerations on degenerate maps}
Note that $\Delta_\circ^{i}$ are cohomology classes of nontrivial sectors of $\PP^1_{a,b,c}$.
We want to show that all the holomorphic orbi-spheres $u : \big(\Sigma,\bz = (z_1, z_2, z_3) \big) \to \PP^1_{a,b,c}$
of appropriate type $\boldsymbol{x} = (x_1, x_2, x_3)$, a triple of twisted sectors of $\PP^1_{a,b,c}$,
can not have any nodal singularity. More precisely, the above ``appropriate'' means that the $\boldsymbol{x}$ is a type with $\text{vir.dim} \,\OL{\CM}_{0, 3}(\PP^1_{a,b,c}, J, \beta, \boldsymbol{x}) = 0$ for all $\beta$. Here, since $\PP^1_{a,b,c}$ is elliptic, the virtual dimension does not depend on $\beta$ (Remark \ref{rmk:c1vanvirdim}).

\begin{lemma}\label{lem:nodal}
There are no degenerate (i.e., nodal) holomorphic orbi-spheres which are non-constant and contribute to $\langle \Delta^i_\circ, \Delta^j_\bullet, \Delta^k_\diamond \rangle^{\PP^1_{a,b,c}}_{0,3}$ for $i + j + k = 1$.
\end{lemma}
\begin{proof}

There are two classes of degenerate maps which are possibly contained in the boundary of the moduli space $\OL{\CM}_{0, 3}(\PP^1_{a,b,c}, J, \beta, \boldsymbol{x})$:
	\begin{enumerate}
		\item $u_1 : \PP^1_{\alpha,\beta,\bullet}\,\sqcup_{\bullet}\, \PP^1_{\bullet, \delta} \to \PP^1_{a,b,c}$,
		\item $u_2 : \PP^1_{\alpha, \beta, \delta,\circ}\,\sqcup_{\circ}  \, \PP^1_{\circ} \to \PP^1_{a,b,c}$,
	\end{enumerate}
where $\bullet$ and $\circ$ are the order of local isotropy group of the nodal point. (See Figure \ref{fig:degen_map}.) Note that $u_i$ ($i = 1,2$) is non-constant map when restricted to the second component of domain, since $u_i$ should be stable. We claim that there can not exist such maps into $\PP^1_{a,b,c}$.

\begin{figure}
\begin{center}
\includegraphics[height=1.9in]{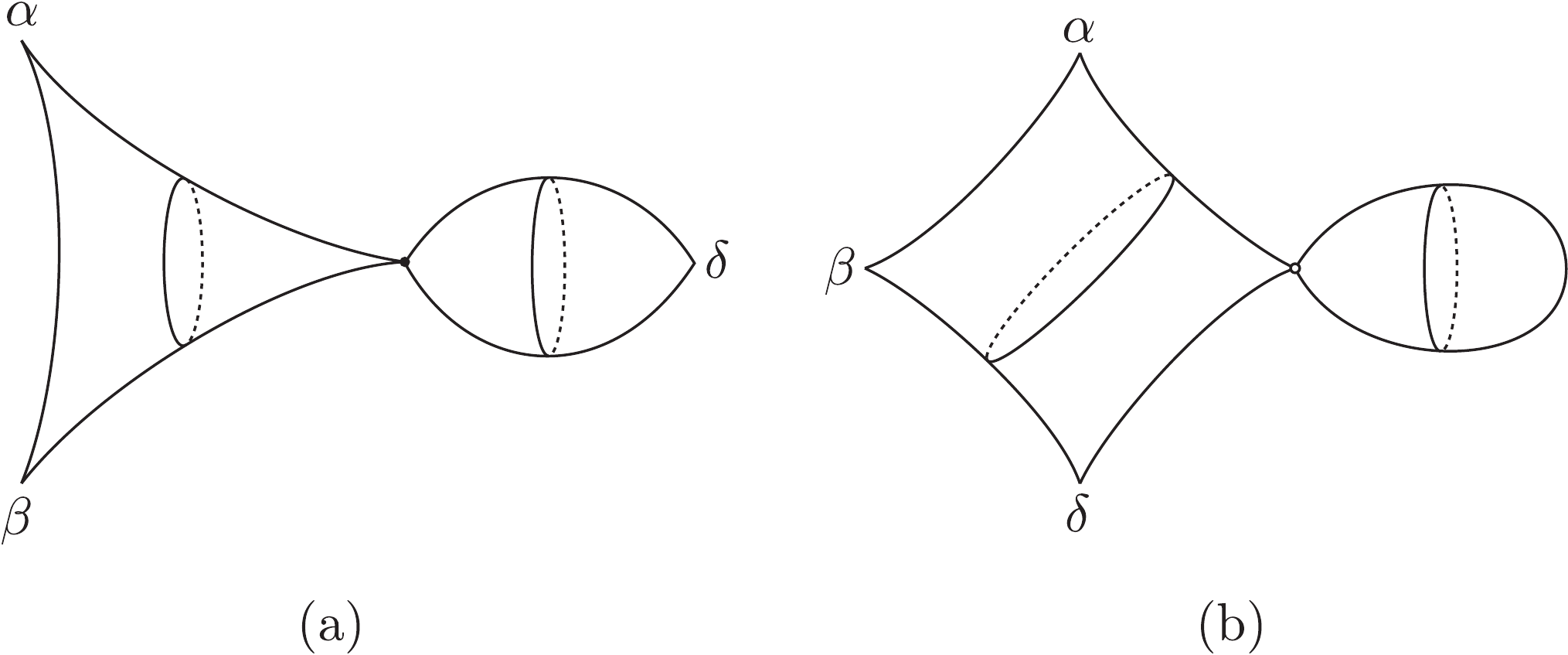}
\caption{Domains of degenerate maps : (a) $\PP^1_{\alpha,\beta,\bullet}\,\sqcup_{\bullet}\, \PP^1_{\bullet, \delta}$ and (b) $\PP^1_{\alpha, \beta, \delta,\circ}\,\sqcup_{\circ}  \, \PP^1_{\circ}$}\label{fig:degen_map}
\end{center}
\end{figure}

First, consider the case of $u : \PP^1_{\delta,\delta} \to \PP^1_{a,b,c}$ ($\bullet = \delta$ in (1)). Since the quotient map $\pi : \PP^1 \to \PP^1_{\delta,\delta}$ (by the obvious $\Z_\delta$-action on $\PP^1$) is holomorphic orbifold covering and $\pi_1 (\PP^1) = 0$, there is a holomorphic map $\WT{u}$ which makes following diagram commutes (using Proposition \ref{prop:cov_lifting} and Lemma \ref{lem:orbimap}).

\begin{equation}
	\xymatrix{
		\PP^1 \ar@{.>}[r]^{\exists \WT{u}} \ar[d]_{\pi} & E \ar[d]^p	\\
		\PP^1_{\delta,\delta} \ar[r]^u & \PP^1_{a,b,c}
	}
\end{equation}

Note that the image of $\WT{u}$ must be homotopic to a constant map since $\pi_2 (E)=0$, and hence $\WT{u}$ is a constant map from the holomorphicity. This contradicts the stability of the map $u$, and hence there is no such holomorphic map $u$. A similar argument shows that $u : \PP^1 \to \PP^1_{a,b,c}$ ($\circ =1$ in (2)) can not exist.

The remaining case is when the second component of the domain is not a good orbifold, $u : \PP^1_{mp,mq} \to \PP^1_{a,b,c}$ for some natural numbers $p$, $q$, and $m$ satisfying $\gcd(p,q)=1$ and $pq \neq 1$. Consider the holomorphic quotient map $\pi : \PP^1_{p,q} \to \PP^1_{mp,mq}$ and the holomorphic map $v := u \circ \pi$. 
\begin{equation*}
	\xymatrix{
		\PP^1_{p,q} \ar[rd]^{v} \ar[d]_{\pi}	\\
		\PP^1_{mp,mq} \ar[r]^u & \PP^1_{a,b,c}
	}
\end{equation*}
Let $x \in \PP^1_{mp,mq}$ be an orbi-singular point and $\WT{x} \in \PP^1_{p,q}$ be the element in $\pi^{-1}(x)$. We may assume that $\WT{x}$ and $u(x)$ have isotropy groups $\Z_p$ for some $p \neq 1$ and $\Z_a$, respectively. Then from the definition of an orbifold map, the map $v$ should be lifted locally to an equivariant map $\WT{v}$ on the local uniformizing charts
\begin{equation}
	\xymatrix{
		U_{\WT{x}} \ar[r]^{\WT{v}} \ar[d] & V_{v(\WT{x})} \ar[d]	\\
		U_{\WT{x}} / \Z_p \ar[r]^v & V_{v(\WT{x})} / \Z_a,
	}
\end{equation}
and the induced group homomorphism between isotropy groups $\phi_{v} : \Z_p \to \Z_a$ should be injective since $\phi_{v} = \phi_{u} \circ \phi_{\pi} : \Z_p \stackrel{\times m}{\to} \Z_{mp} \to \Z_a$ is a composition of two injective morphism. (The injectivity of the second map comes from the definition of an orbifold map. See Definition \ref{def:smoothbetorb}.) 
Hence the generator $g$ of $\pi^{orb}_1 (U_{\WT{x}} / \Z_p )$ should be mapped to an order $p$ element of $\pi^{orb}_1 (V_{v(\WT{x})} / \Z_a )$. However, from the van Kampen's theorem, $\pi^{orb}_1 (\PP^1_{p,q}) = \{0\}$, so the image of $g$ in $\pi^{orb}_1 (\PP^1_{p,q})$ is zero whereas  $\pi^{orb}_1 (\PP^1_{a,b,c}) = \left\langle g_1, g_2, g_3 \, \big{|} \, g_1^a= g_2^b = g_3^c= g_1 g_2 g_3 = 1 \right\rangle$ is nontrivial. Note that the homomorphism $\pi^{orb}_1 (V_{v(\WT{x})} / \Z_a ) \to \pi^{orb}_1 (\PP^1_{a,b,c})$ induced from the inclusion map $\iota: V_{v(\WT{x})} \to \PP^1_{a,b,c}$ is injective, since $\PP^1_{a,b,c}$ is a good orbifold (See for example \cite[Prop.1.18]{D}).  This gives a contradiction.
\end{proof}

Consider an orbifold stable map $\big(f, (\PP^1, \bz), \xi \big)$ with three makings of type $\boldsymbol{x}$.
If there is a smooth point in $\bz$, $f$ can be thought of as a map from an orbi-sphere with two singular points. Then exactly the same argument in the proof of Lemma \ref{lem:nodal} implies that $f$ is indeed a constant map.

\subsection{Orbifold coverings of $\mathbb{P}^1_{a,b,c}$ contributing to the quantum product}\label{subsec:orbsphorbcover}
In this section, we prove that holomorphic orbi-spheres satisfying certain properties become orbifold covering maps. Most of holomorphic orbi-spheres contributing to the quantum product of elliptic $\mathbb{P}^1_{a,b,c}$ will turn out to satisfy these properties, later. (There is only one  exceptional case for $\PH$ where non-trivial holomorphic orbi-spheres from the hyperbolic orbifold $\PP^1_{3,6,6}$ contribute to the quantum product of $\PH$.)

Let $u$ be a holomorphic orbi-sphere from $\PP^1_{\alpha, \beta, \delta}$ to $\PP^1_{a,b,c}$ and consider the universal covering map $\pi : \WT{\PP^1_{\alpha, \beta, \delta}} \to \PP^1_{\alpha, \beta, \delta}$ and $p : \C \to \PP^1_{a,b,c}$. Here, $p$ is a holomorphic map since the complex structure on $\PP^1_{a,b,c}$ comes from that on its universal cover. 
From orbifold covering theory, we obtain a lifting $\WT{u}$ of $u \circ \pi$ of the underlying orbifold morphism $u : \PP^1_{\alpha, \beta, \delta} \to \PP^1_{a,b,c}$:
\begin{equation}\label{liftingmap}
	\xymatrix{
		\WT{\PP^1_{\alpha, \beta, \delta}} \ar@{.>}[r]^{\exists \WT{u}} \ar[d]^{\pi} \ar[dr]^{u\circ \pi} & \C \ar[d]^{p} \\
		\PP^1_{\alpha, \beta, \delta} \ar[r]^{u} & \PP^1_{a,b,c}.
		}
\end{equation}
For each equivalence class $[u] \in \CM^{\text{reg}}_{0,3,\beta} (\PP^1_{a,b,c})$, if we choose a representative $u$ of $[u]$ by fixing the location of three special points on the domain (denoted by $\PP^1_{\alpha, \beta, \delta}$), there is no further equivalence relation since there is a unique automorphism which sends given three point to the other. For such $u$, the lifting $\WT{u}$ is holomorphic, since it is locally holomorphic. 

To avoid notational complexity, let us write ${X}$ for $\PP^1_{a,b,c}$, and consider a triple of twisted sectors $\boldsymbol{x} = ( {X}_{( g_1 )}, {X}_{( g_2 )}, {X}_{( g_3 )} )$. Let $G_i$ be an isotropy group of a point in ${X}_{(g_i)}$ which is defined up to conjugacy.
Since ${X}$ is one dimensional, the age of an element of $G_i$ is given by $\iota_{(g_i)} = \frac{l_i}{|G_i |}$ for some $l_i \in \{1, \cdots, |G_i |-1\}$. Since we only count 0-dimensional strata of moduli of orbifold stable maps for the quantum product, we assume that
	\begin{equation}\label{eq:vdimzero}
		\sum_{i=1}^{3} \iota_{(g_i)} = 1.
	\end{equation}

From Definition \ref{def:orbcover}, we see that the necessary condition for $u$ to be an orbifold covering map is that
	\begin{equation}\label{eq:orb-insert}
		\iota_{(g_i )}^{-1} \in \Z \quad i = 1,2,3,
	\end{equation}
or equivalently $l_i \mid |G_i|$ for $i=1,2,3$.
Nonetheless, this condition \eqref{eq:orb-insert} is indeed sufficient to guarantee that $u$ is an orbifold covering map:

\begin{lemma}\label{lem:u_orb_cover}
If $u$ is a non-constant holomorphic orbifold stable map from $(\PP^1, \bz, \bm)$ to ${X}$ of the type $\boldsymbol{x}$ satisfying \eqref{eq:vdimzero} and \eqref{eq:orb-insert}, then $u$ is an orbifold covering map. Here, $\bz = (z_1, z_2, z_3)$ is a triple of marked points, and $\bm=(m_1, m_2, m_3)$ is the triple of orders of isotropy groups of $\bz$.
\end{lemma}
\begin{proof}
Recall that any non-constant holomorphic map between Riemann surfaces is a branched covering. 
We want to show that this map is an orbifold covering (See Definition \ref{def:orbcover}).

Consider three orbi-singular points $\{ w_1, w_2, w_3 \}$ in the target space and their inverse image $u^{-1}(w_i)$.
Since $\boldsymbol{x}$ consists of twisted sectors, there is a function $I : \{1,2,3\} \to \{1,2,3\}$ such that $u(z_i) = w_{I(i)} \in {X}_{(g_i)}$.
We denote the number of points in $u^{-1}(w_i) - \{z_1, z_2, z_3\}$ by  $m(w_i)$. 

Note that each $z_i$ has degree $3a_i + l_i$ for some $a_i \in \N_{\geq 0}$ where $l_i = \iota_{( g_i )}  | G_i |$.
The other points in the inverse image $u^{-1} (w_1)$ has local degree which is a multiple of $a$, say $a e_j$ for some $e_j \in \N$ ($j = 1, \cdots, m(w_1)$). 
Similarly, in $u^{-1}(w_2)$ and $u^{-1}(w_3)$,
there are $m(w_2)$ and $m(w_3)$ number of points with degree $b f_k$ and $c g_l$ for $1 \leq k \leq m(w_2) $ and $1 \leq l \leq m(w_3)$, respectively.

Since any orbifold Riemann surface $(\Sigma_g, \bz)$ is analytically isomorphic to a smooth Riemann surface $\Sigma_g$, $u$ can be regarded as a branched covering map between two $\CP^1$'s. In particular, one can apply the Riemann-Hurwitz formula to $u$ so that
\begin{align}\label{eq:RHur}
2 \leq 2d - \left\{ \sum_{i=1}^3 (|G_i| a_i + l_i - 1) + \sum_{j=1}^{m(w_1)}(a e_j - 1) + \sum_{k=1}^{m(w_2)}(b f_k - 1) + \sum_{l=1}^{m(w_3)}(c g_l - 1) \right\}
\end{align}
where $d$ is the degree of $u$. (Here, $2$ in the left hand side is the topological Euler characteristic of $\PP^1_{a,b,c}$.) If $u$ does not have any branching outside $u^{-1}(w_1) \cup u^{-1}(w_2) \cup u^{-1}(w_3)$, then the equality holds in \eqref{eq:RHur}. Since $d$ is the weighted count of the  number of points in the fiber $u^{-1} (w_i)$ of $u$, we have
\begin{equation}\label{eq:ddeginv}
\begin{array}{rl}
d 	&= \displaystyle\sum_{i \in I^{-1}(1)}(|G_i| a_i + l_i ) + a \sum_{j = 1}^{m(w_1)} e_j = \sum_{i \in I^{-1}(2)}(|G_i| a_i + l_i ) + b \sum_{k = 1}^{m(w_2)} f_k \\
&= \displaystyle\sum_{i \in I^{-1}(3)}(|G_i| a_i + l_i ) + c \sum_{l = 1}^{m(w_3)} g_l,
\end{array}
\end{equation}
and hence by inserting \eqref{eq:ddeginv} in the \eqref{eq:RHur},
\begin{equation}\label{eq:ddeginv1}
d \leq 1 + \sum_{j = 1}^{3} m(w_j).
\end{equation}

\eqref{eq:ddeginv} together with $e_j \geq 1$ implies that
\begin{align}\label{eq:ddeginv2}
d \geq \sum_{i \in I^{-1}(1)}(|G_i| a_i + l_i ) + a m(w_1),
\end{align}
and similar inequalities also hold for $w_2$ and $w_3$.

Combining \eqref{eq:ddeginv1} and \eqref{eq:ddeginv2},
\begin{equation}\label{eq:main}
\begin{array}{l}
abc \left( \displaystyle\sum_{i=1}^3 m(w_i) \right)+ bc \displaystyle\sum_{i \in I^{-1}(1)} \big(|G_i| a_i + l_i \big) \\+ ca \displaystyle\sum_{i \in I^{-1}(2)}\big(|G_i| a_i + l_i \big)  + ab \displaystyle\sum_{i \in I^{-1}(3)}\big(|G_i| a_i + l_i \big) \\ 
	\leq (bc + ca + ab) d = abc d
\end{array}
\end{equation}
where the last equality follows from $\frac{1}{a} + \frac{1}{b} + \frac{1}{c} = 1$.

Note that $\sum_{i =1}^3 \frac{l_i}{|G_i|} = 1$ from the condition \eqref{eq:vdimzero}.
Hence
\begin{align}\label{eq:vdimzero1}
bc \sum_{i \in I^{-1}(1)} l_i + ca \sum_{i \in I^{-1}(2)} l_i+ ab \sum_{i \in I^{-1}(3)} l_i = abc.
\end{align}

Since $|G_i| = a$ if $i \in I^{-1}(1)$, (and with similar equalities for other two cases)
\begin{align}\label{eq:aux}
bc \sum_{i \in I^{-1}(1)} (|G_i| a_i) + ca \sum_{i \in I^{-1}(2)}(|G_i| a_i) + ab \sum_{i \in I^{-1}(3)}(|G_i| a_i) = abc \sum_{i = 1}^3 a_i.
\end{align}

Combining \eqref{eq:main} with \eqref{eq:ddeginv1}, \eqref{eq:vdimzero1}, and \eqref{eq:aux}, it follows that
\begin{align}\label{eq:comb}
abc \left( \sum_{i=1}^3 m(i) + \sum_{i = 1}^3 a_i  + 1 \right) \leq abc \left( 1 +  \sum_{j = 1}^{3} m(w_j) \right),
\end{align}
hence $a_1 = a_2 = a_3 = 0$ from the non-negativity of $a_i$'s.

If we do not use the inequality \eqref{eq:ddeginv2} and proceed, we have the following more precise estimate 
\begin{align*}
abc \left( \sum_{j=1}^{m(w_1)} e_j + \sum_{k=1}^{m(w_2)} f_k + \sum_{l=1}^{m(w_3)} g_l + \sum_{i = 1}^3 a_i + 1 \right) \leq abc \left( 1 +  \sum_{j = 1}^{3} m(w_j) \right)
\end{align*}
which implies $e_j = f_k = g_l =1$ for all $j,k,l$. Therefore, $u$ is an orbifold covering.
\end{proof}

\begin{remark}
For the case of $c_1 (T{{X}}) < 0$ (i.e., ${X} = \PP^1_{a,b,c}$ is hyperbolic), the same lemma also holds since \eqref{eq:main} is still valid.
However, if $c_1 (T {X}) > 0$ (i.e., ${X}$ is spherical), the argument in Lemma \ref{lem:u_orb_cover} is not true any more.

%
\end{remark}

\subsection{Regularity of holomorphic maps}

Finally, we show that holomorphic orbi-spheres which become orbifold coverings of elliptic $\mathbb{P}^1_{a,b,c}$ are Fredholm regular.                                                                                                                                                                                                                                                                                                                                                                                                                                                                                                                                                                                                                                                                                                                                                                                                                                                                                                                                                                                                                                                                                                                                                                                                                                                                                                                                                                                                                                                                                                                                                                     We first recall the definition of the desingularization of an orbi-bundle, and examine some  properties of the desingularized bundle which will be used for proving the regularity.

Let $(\Sigma, \bz, \bm)$ be an orbi-Riemann surface and consider an orbi-bundle $E \to (\Sigma, \bz, \bm)$. The desingularization of $E$ is defined as follows.
For each disc neighborhood $D_i$ of orbi-singular points $z_i$ in $\bz$, $E$ can be uniformized by $(D_i \times \C^n, \Z_{m_i}, \pi)$ so that the action is linear and diagonal. Hence the action can be written as
\begin{equation}
	e^{\frac{2\pi \sqrt{-1}}{m_i}} \cdot \left( z, f \right) = \left( e^{\frac{2\pi \sqrt{-1}}{m_i}} z, \text{diag} \left(e^{\frac{2\pi \sqrt{-1} m_{i,1}}{m_i}}, \cdots, e^{\frac{2\pi \sqrt{-1} m_{i, n}}{m_i}}\right) f \right)
\end{equation}
for some integers $0 \leq m_{i,j} < m_i$ ($j = 1, \cdots, n$).
Let $\Phi_i :  D^* \times \C^n \to D^* \times \C^n$ be  a $\Z_{m_i}$-equivariant map over the punctured disc $D^* = D_i -\{0\}$ defined by
\begin{equation}\label{eq:localdesing}
	\Phi_i (z, f_1, \cdots, f_n) = \left(z^{m_i}, z^{-m_{i,1}} f_1, \cdots,  z^{-m_{i,n}} f_n\right)
\end{equation}
where the $\Z_{m_i}$ trivially acts on the right side. 
Consider the natural map $\phi : (\Sigma, \bz, \bm) \to |\Sigma|$ which can be written as $z \to z^{m_i}$ over each $D_i$. Then the local holomorphic chart on $|\Sigma|$ is $w = z^{m_i}$ over each $\phi (D_i)$.
We construct a complex vector bundle $|E|$ over the underlying space $|\Sigma|$ of $(\Sigma, \bz, \bm)$ by extending the complex vector bundle over $|\Sigma| - \{ z_1, \cdots z_k \}$ whose trivialization is given by the right hand side of \eqref{eq:localdesing}.

Chen and Ruan \cite{CR1} observed that the first Chern number of an orbi-bundle is the sum of the first Chern number of its de-singularization and the ages of representations which are induced from local trivialization of the orbi-bundle over orbi-singular points.
More precisely, let $E$ be an orbi-bundle over a closed orbi-Riemann surface $(\Sigma, \bz, \bm)$ and set $\bm = (m_1, \cdots, m_k)$ for $m_i \in \Z_{\geq 0}$. Then for each orbi-singular points $z_i$, the induced representation $\rho_i :\Z_{m_i} \to \text{End}(\C^n)$ can be written as
\begin{equation*}
	\rho_i \left(e^{\frac{2\pi \sqrt{-1}}{m_i}} \right) = \text{diag} \left(e^{\frac{2\pi \sqrt{-1} m_{i,1}}{m_i}}, \cdots, e^{\frac{2\pi \sqrt{-1} m_{i, n}}{m_i}}\right)
\end{equation*}
for some integers $0 \leq m_{i,j} < m_i$ ($j = 1, \cdots, n$).
Then 
\begin{equation}\label{eq:desingbchern}
	c_1 (E) ([\Sigma]) = c_1 (|E|) ([\Sigma]) + \sum_{i=1}^{k} \sum_{j=1}^{n} \frac{m_{i,j}}{m_i}.
\end{equation}

For a holomorphic orbi-bundle $E \to {X}$, let us denote sheaves of holomorphic sections of $E$ and $|E|$ over ${X}$ and $|{X}|$ by $\mathcal{O}(E)$ and $\mathcal{O}(|E|)$, respectively. 
Then we have $\mathcal{O}(E) = \mathcal{O}(|E|)$ \cite[Proposition 4.2.2]{CR1} from the removability of isolated singularities of $J$-holomorphic maps. In detail, if $g : D_i \to \C^n$ is a local holomorphic section over $|E| \big|_{D_i} \to D_i$, then $g(w) = ( g_1 (w), \cdots, g_n (w) )$ for some holomorphic maps $g_i : D_i \to \C$ with respect to the trivialization taken as above. If we pullback this section via $\Phi_i$, then the corresponding section on $E|_{D_i} \to D_i$ is the holomorphic map $f : D_i \to \C^n$ whose components are $f_j (z) = z^{m_{i,j}} g_j(z^{m_i})$ for each $j = 1, \cdots, n$.

Conversely, let $f = (f_1, \cdots f_n) : D_i \to \C^n$ be a given local holomorphic section on orbi-bundle $E|_{D_i} \to D_i$, i.e., $f$ is a $\Z_{m_i}$-equivariant holomorphic section. Define a map $g : D^* \to \C^n$ whose components are $g_j(w) := z^{-m_{i,j}} f_j(z)$ for $z = w^{\frac{1}{m_i}}$. Note that the $\Z_{m_i}$-equivariantness of $f$ says that the section $g$ is well-defined, although there is an ambiguity on the choice of brach cut for $z = w^{\frac{1}{m_i}}$. Moreover, if we expand the holomorphic function $f_j$ as
$$
f_j(z) = \sum_{n=0}^\infty a_{n,j} z^n,
$$
then $a_{n,j} = 0$ unless $n \equiv m_{i,j}$ modulo $m_i$. Thus, the $\left| \frac{f_j(z)}{z^{m_{i,j}}} \right|$ is bounded on $D_i$, and we conclude that $g$ can be extended to a holomorphic section over $D_i$ using the Riemann extension theorem.
One can easily check that this process is the inverse of the pullback via $\Phi_i$, which gives $\CO(E) = \CO(|E|)$.
%
%

Now we prove the Fredholm regularity of holomorphic orbi-spheres in elliptic $\mathbb{P}^1_{a,b,c}$, which are orbifold coverings.

\begin{prop}\label{lem:reg}
Let ${X}$ be an elliptic orbi-sphere and $(\Sigma, \bz)$ be a domain Riemann orbi-curve. If $u : \Sigma \to {X}$ is a holomorphic orbi-sphere satisfying \eqref{eq:vdimzero} and \eqref{eq:orb-insert}, then $u$ is regular.
\end{prop}

\begin{proof}
Consider the pullback orbi-bundle $u^*T{X} \to \Sigma$ and the linearized $\dbar$-operator $D\dbar_{J} (u) : C^\infty (u^* T{X}) \to \Omega^{0,1}(u^* T{X})$.
Since $J$ is integrable, $D\dbar_{J}(u) = \dbar_{J}$. Hence it is sufficient to show that $H^{0,1}_{\dbar}({X},u^*T{X})=0$.

Note that the first Chern number of the tangent bundle of ${X}$ is zero, and for any orbifold covering $u$, 
that of $u^*T{X}$ is also zero.
From this, we can see that the desingularized bundle of $u^*T{X}$, $|u^*T{X}|$ has (desingularized) Chern number $-1$ since the second term in the right hand side of \eqref{eq:desingbchern} is $1$ from \eqref{eq:vdimzero}. i.e. $c_1 (|u^*T{X}|) = -1$.
Since ${X}$ is a complex orbifold and $u$ is a holomorphic orbi-map, $u^* T{X}$ is a holomorphic orbi-bundle over $\Sigma$.
The desingularization of holomorphic orbi-bundle is also holomorphic.
From the Lemma 3.5.1 in \cite{MS}, this implies that the holomorphic line bundle $|u^*T{X}|$
has vanishing cohomology group $H^{0,1}_{\dbar}(|\Sigma|,|u^*T{X}|)$.
As the sheaf of holomorphic sections of $|u^*T{X}|$ is the same as the sheaf of (orbifold) holomorphic sections of 
$u^*T{X}$ on $\Sigma$, we have the vanishing of $H^{0,1}_{\dbar}(\Sigma,u^*T{X})$.
More precisely,
\begin{align*}\label{dol-sheaf}
H^{0,1}_{\dbar}(|\Sigma|,|u^*T{X}|) &\cong H^1 \big(|\Sigma|,\mathcal{O}(|u^*T{X}|)\big) \\
						&\cong H^1 \big(\Sigma, \mathcal{O}(u^* T{X})\big) \\
						&\cong H^{0,1}_{\dbar}(\Sigma,u^*T{X}).
\end{align*}
For the last isomorphism, note that the two term complex $\dbar_J : C^{\infty}(u^*T{X}) \to \Omega^{0,1}(u^*T{X})$ is a fine resolution for the sheaf of holomorphic sections $\mathcal{O}(u^*T{X})$ as in the smooth case.
\end{proof}

\begin{remark}
Even if $u$ is not an orbifold covering, we still have $c_1 (u^* T{X}) \geq  c_1 (T{X})$. (Indeed, we can improve this inequality by considering the degree of $u$.) Therefore, the above proposition also holds as long as ${X}$ has a non-negative first Chern number. 

For example, when we calculate the quantum cohomology of $\PH$, we need to count holomorphic orbi-spheres $u : \PP^1_{3,6,6} \to \PH$ which are not orbifold covering maps. These orbi-spheres are also Fredholm regular exactly by the same argument.
\end{remark}

%
%

\section{The quantum cohomology ring of $\PT$}\label{sec:qhPT}
%
%

In this section, we explicitly compute the product structure on $QH^\ast_{orb} (\PT, \Q)$, which proves Theorem \ref{thm:main1}. For this, we first classify holomorphic orbi-spheres in $\PT$ (Section \ref{subsec:classorspPT}). Recall from Lemma \ref{lem:nodal} and Proposition \ref{lem:reg} that these stable maps, in fact, are maps from a single orbi-sphere component and are regular. Thus by counting holomorphic orbi-spheres inside $\PT$, we obtain the $3$-fold Gromov-Witten invariant for $\PT$ which combined with the constant map contributions (Section \ref{subsec:contriconst}) gives rise to the quantum product on $QH^\ast_{orb} (\PT, \Q)$. One interesting feature is that one can relate these orbi-spheres with the solutions of a certain Diophantine equation. 
%
%
%

It will turn out in Section \ref{subsec:classorspPT} that only
$$\langle \da{1}{3}, \db{1}{3}, \dc{1}{3} \rangle^{\PT}_{0,3} \quad \mbox{and} \quad \langle \Delta_{i}^{1/3}, \Delta_{i}^{1/3}, \Delta_{i}^{1/3} \rangle^{\PT}_{0,3}$$ 
for $i=1,2,3$ are nontrivial, which precisely give the coefficients $f_0$ and $f_1$ of cubic terms for the Gromov-Witten potential in \cite[Theorem 3.1]{ST}.

\begin{remark}
We will write the details on the classification orbi-spheres in $\PT$ as concrete as possible.
 For other two cases, $\PH$ and $\PQ$, we will find similar classification results in Section \ref{sec:236244}, but without much details as the arguments are not very much different from the one for $\PT$.
\end{remark}

%

\subsection{Classification of orbi-spheres in $\PT$}\label{subsec:classorspPT}
From the expected dimension formula and representability of orbifold stable map, the only possible domain orbi-sphere in 0-dimensional components of the moduli space $\OL\CM_{0,3,\beta}(\PT) $ is $\PT$ itself.
Since there exists a unique biholomorphism $\phi : (\PT, \bz) \to (\PT, \bz')$ sending any triple of orbi-points $\bz = (z_1, z_2, z_3)$ to another $\bz' = (z'_1, z'_2, z'_3)$, there is no domain parameter in the relevant moduli $\OL{\CM}_{0,3,\beta}(\PT)$.
Hence from now on, we take the domain orbi-sphere to be $\PT$ with the fixed conformal structure which is induced by the quotient map $E \to \PT$ in the Section \ref{subsec:GWforP}.

By degree reason, $\langle \Delta_\circ^i, \Delta_\bullet^j, \Delta_\diamond^k \rangle^{\PT}_{0,3}$ is trivial unless $i+j+k=1$. 
Note that by the obvious symmetry on $\PT$, it is enough to consider only the following three cases:
\begin{enumerate}
\item $\langle \da{1}{3}, \da{1}{3}, \da{1}{3} \rangle_{0,3}$,
\item $\langle \da{1}{3}, \db{1}{3}, \dc{1}{3} \rangle_{0,3}$,
\item $\langle \da{1}{3}, \db{1}{3}, \db{1}{3} \rangle_{0,3}$.
\end{enumerate} 
Namely, we may assume that the first marked point $z_1$ in the domain orbi-sphere is mapped to the orbi-singular point $w_1$ in $\PT$ associated with $\da{1}{3}$.

Let $u$ be a holomorphic orbi-sphere from $\PT$ to $\PT$, which contributes to $\langle \da{1}{3}, \Delta_\circ^{1/3}, \Delta_\bullet^{1/3} \rangle_{0,3}$.
 We fix base points of the domain orbi-sphere and the target orbi-sphere of $u$ and their universal coverings as follows: $x_0:=z_1 \in \PT$ and $\WT{x_0}:=0 \in p^{-1}(x) (\subset \C)$ for the domain $\PT$, and $y_0:=w_1 \in \PT$ and $\WT{y_0}:=0 \in p^{-1}(w_1) (\subset \C)$. 
Recall that $p : \C \to \PT$ is the orbifold universal covering, and $u(x_0) = y_0$. 
Thus, we obtain a unique lifting $\WT{u}$ of $u \circ p$ for the underlying holomorphic orbi-sphere $u : \PT \to \PT$:
\begin{equation}\label{liftingmap}
	\xymatrix{
		\C \ar@{.>}[r]^{\exists \WT{u}} \ar[d]^{p} & \C \ar[d]^{p} \\
		\PP^1_{3,3,3} \ar[r]^{u} & \PP^1_{3,3,3}.
		}
\end{equation}

Note that the conditions in the lemma \ref{lem:u_orb_cover} are automatic in this case. Therefore, $u : \PT \to \PT$ is an orbifold covering, and its lifting $\WT{u}$ \eqref{liftingmap} has  a particularly nice shape.

\begin{prop}\label{prop:lambdaz}
If $u$ is a non-constant holomorphic orbi-sphere contributing to $\langle \da{1}{3}, \Delta_\circ^{1/3}, \Delta_\bullet^{1/3} \rangle_{0,3}$,
 then $\WT{u}(z) = \lambda z$ for some $\lambda \in \Z[\tau]$,
 where $\Z[\tau] := \{ a + b \tau \,\, | \,\, a, b \in \Z \}$ 
\end{prop}


\begin{proof}
Because $u$ is an orbifold universal covering, so is the composition $u \circ p$. Now, by the uniqueness of orbifold universal covering, $\WT{u}$ should be a homeomorphism. 
Note that $\WT{u}$ is an entire proper holomorphic map, since $\WT{u}$ is a homeomorphism and is a lifting of the holomorphic map $u \circ p$. It is well-known that any entire and proper holomorphic map on $\C$ is a polynomial. Since $\WT{u}$ is invertible, we conclude that $\WT{u}$ is a linear map $\WT{u}(z) = \lambda z $ for some $\lambda \in \C$. Here, $\WT{u}$ does not have a constant term because the lifting preserves the base points, $\WT{u} (\WT{x_0}) = \WT{y_0}$ (i.e., $\WT{u} (0) = 0$).

\begin{figure}
\begin{center}
\includegraphics[height=1.8in]{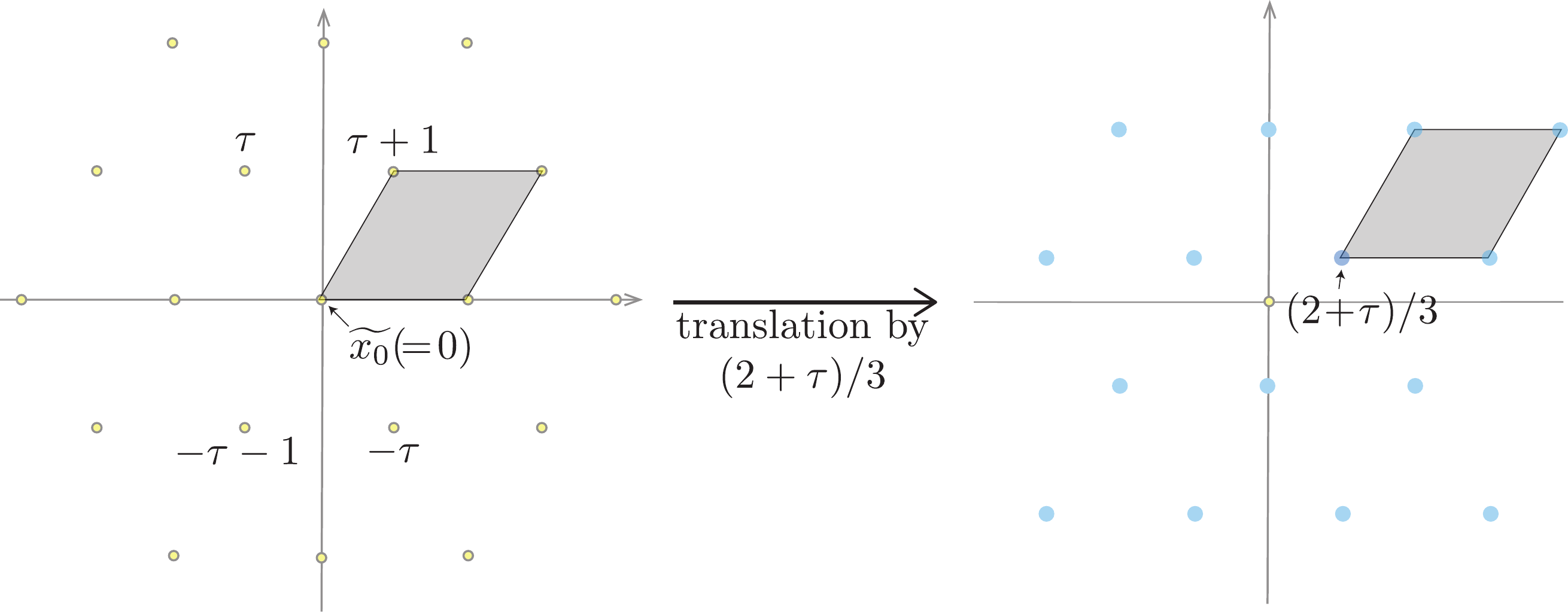}
\caption{The shaded regions show the images of degree-3 maps contributing to $\langle \da{1}{3}, \da{1}{3}, \da{1}{3}  \rangle_{0,3}$ (LHS) and $\langle \db{1}{3}, \db{1}{3}, \db{1}{3}  \rangle_{0,3}$ (RHS)}\label{fig:Ztau}
\end{center}
\end{figure}
%

Recall from Subsection \ref{subsec:GWforP} that $p^{-1}(x_0)$ gives a lattice $\Z \langle 1 ,\tau \rangle = \Z [\tau]\, (\because \tau^2 = -\tau -1) $ in $\C \cong \R^2$. (See the left side of Figure \ref{fig:Ztau}.) Since $1 \in \C$ is an element of this lattice, $\WT{u}$ should map $1$ to a point in $p^{-1} (y_0)$ which is also the same lattice $\Z[\tau]$ in $\C$. Thus, $\lambda= \WT{u} (1)$ has to lie in $\Z[\tau]$, which finishes the proof.
%
%
\end{proof}

It is obvious from the picture that the fundamental domain of the domain orbi-sphere covers that of the target orbi-sphere $|\lambda|^2$-times by the holomorphic map induced by the linear map $z \mapsto \lambda z$ between the universal covers. Another way to see this is to consider the energy $\int |du|^2$ which equals the symplectic area of the holomorphic orbi-sphere $u$. Note that for $\tilde{u} (z) =\lambda z$, $|d \tilde{u}|^2 = |\lambda|^2$. Therefore, such a map induces a term containing $q^{|\lambda|^2}$ in the (3-fold) Gromov-Witten potential.

Conversely, any linear map $\WT{u} = \lambda z$ with a coefficient $\lambda$ in  $\Z[\tau]$ induces a $\Z_3$-equivariant holomorphic map between the middle level torus $E$. The equivariance implies that $\WT{u}$ descends to a holomorphic map $u : \PT \to \PT$. $u$ is a well-defined orbifold morphism between $\PT$, as it is represented by an equivariant map between $E$, and $\PT =  [E / \Z_3]$.

Later, we will establish the one-to-one correspondence between linear maps with $\Z [\tau]$-coefficients modulo ``certain equivalences" and orbi-spheres in $\PT$ contributing to $\langle \da{1}{3}, \Delta_\circ^{1/3}, \Delta_\bullet^{1/3} \rangle_{0,3}$
modulo equivalences given in \eqref{def:GWgauge_equiv}.


\begin{remark}\label{gpoidmap}
Let $\sigma$ be the cyclic permutation $(1,2,3)$ in the permutation group $S_3$ on $3$-letters $\{1,2,3\}$. If $u$ contributes to $\langle \da{1}{3}, \Delta_i^{1/3}, \Delta_j^{1/3} \rangle_{0,3}$, then it is easy to see from Figure \ref{fig:Ztau} that the translation of $\WT{u}$ by an element of $\Z \langle \frac{2+\tau}{3},\frac{1+2 \tau}{3}\rangle$ induces a map $v : \PT \to \PT$ contributing to $\langle \Delta_{\sigma^k (1)}^{1/3}, \Delta_{\sigma^k (i)}^{1/3}, \Delta_{\sigma^k (j)}^{1/3} \rangle_{0,3}$ for some $0 \leq k \leq2$. This explains the coincidence between various $3$-fold Gromov-Witten invariants appearing in \cite[Theorem 3.1]{ST}.
Note that $\Z [\tau]$ is a sub-lattice of $\Z \langle \frac{2+\tau}{3},\frac{1+2 \tau}{3}\rangle$.
\end{remark}

\subsection{Symmetries of the lifting of orbi-maps}\label{subsec:symm333}


We have proved that any orbi-spheres with orbi-insertion  $\da{1}{3}, \Delta_i^{1/3}, \Delta_j^{1/3}$ can be lifted to a linear map $z \mapsto \lambda z$ for $\lambda \in \Z [\tau]$.
We next investigate a natural equivalence relation $\sim$ on $\{ \WT{u} = \lambda z \, | \, \lambda \in \Z [\tau] \}$ such that if $\WT{u}_1 \sim \WT{u}_2$, then these map induce a pair of equivalent orbi-spheres.
Let us denote the set of linear maps $\{ z \mapsto \lambda z | \lambda \in \Z[\tau] \}$ by $L(\C)$.
We now find the equivalence relation on $L(\C)$ such that the set of equivalence classes of $L(\C)$ corresponds  bijectively to the moduli space $\OL{\CM}_{0,3} (\PT;\da{1}{3}, \Delta_i^{1/3}, \Delta_j^{1/3})$ for $1 \leq i,j \leq 3$.

%

Recall that positions of three orbi-markings as well as the domain $\PT$ itself are fixed by regarding $\PT$ as a quotient of $E=\C/\Z\langle 1, \tau \rangle$ via the $\Z/3$-action.
Therefore, we do not have an equivalence from a domain reparametrization and, it is enough to find the condition for two linear maps $\WT{u}_i (z) = \lambda_i z$ $\big(\lambda_i \in \Z[\tau], (i=1,2)\big)$ inducing the same map on the quotient orbifold.

Denote the induced orbi-spheres from $\WT{u}_i$ by $u_i$ for $i=1,2$, and suppose that they have orbi-insertions $\da{1}{3}$, $\Delta_i^{1/3}$ and $\Delta_j^{1/3}$ at $z_1$, $z_2$ and $z_3$, respectively.
Since $u_1$ and $u_2$ are the same orbifold morphism and both of them send $z_1$ to $w_1$, their local liftings should be related by the local isotropy group at $w_1$, which is isomorphic to $\Z_3$ and is generated by the $\tau$-multiplication. On the level of universal covers, this local group can be realized as the local isotropy group at the origin (lying in $p^{-1} (w_1)$) of $\C$. Local groups at other points in the fiber $p^{-1} (w_i)$ can not relate $\WT{u}_1$ and $\WT{u}_2$ since they do not preserve the origin.
%
Consequently, $\lambda_1 = \tau^k \lambda_2$ for some $k = 0, 1, 2$, and this gives the desired equivalence relation on $L(\C)$.

For computational simplicity, we consider another type of symmetry on $\PT$, which is induced from the $(1 + \tau)$-multiplication on $\C$ (Note that $(1 + \tau)^2 = \tau$).
This action gives rise to an action of $\Z/6$ on $L(\C)$ in an obvious way. In view of holomorphic orbi-spheres corresponding to elements of $L(\C)$, this action switches two orbi-insertions $\Delta_i^{1/3}$ and $\Delta_j^{1/3}$ without changing the degree.
Thus, the $(1+\tau)$-multiplication gives the one-to-one correspondence
$$\OL{\CM}_{0,3,d} (\PT;\da{1}{3}, \db{1}{3}, \dc{1}{3}) \longleftrightarrow \OL{\CM}_{0,3,d} (\PT;\da{1}{3}, \dc{1}{3}, \db{1}{3}).$$

In Section \ref{subsec:compP1333}, we will count elements in $L(\C)$ whose underlying holomorphic orbi-spheres contributing $\langle \da{1}{3}, \db{1}{3}, \dc{1}{3} \rangle_{0,3}$ or $\langle \da{1}{3}, \dc{1}{3}, \db{1}{3} \rangle_{0,3}$ simultaneously.
Then, dividing the number of such linear maps by the order of the group generated by the $(1+\tau)$-multiplication which is $6$, we find the presentation of 
$$f_0(q) = \sum_{d \in H_2 ({X}, \Z)} \langle \da{1}{3}, \db{1}{3}, \dc{1}{3} \rangle^{{X}}_{0,3,d} q^d .$$

\subsection{Identification of inputs}\label{subsec:compP1333}
We have shown that a degree-$d$ non-constant holomorphic orbi-sphere in $\PT$ contributing to $\langle \da{1}{3}, \Delta_i^{1/3}, \Delta_j^{1/3} \rangle_{0,3,d}$ has one-to-one correspondence with a linear map $z \mapsto \lambda z$ for some $\lambda = a + b \tau \in \Z[\tau]$ with $|\lambda|^2 =d (\neq 0)$. (Recall that this $d$ is really the degree of the corresponding holomorphic orbi-sphere. See the discussion below the proof of Proposition \ref{prop:lambdaz}.)
We subdivide this set of holomorphic orbi-spheres in terms of their orbi-insertions. Orbi-insertions of the holomorphic orbi-sphere corresponding to $z \mapsto \lambda z$ can be determined in the following way. Note that the triangle with vertices $0$, $\frac{1+2\tau}{3}$ and $\frac{2+\tau}{3}$ in the universal cover of the domain $\PT$ gives a fundamental domain for the upper hemisphere of $\PT$. Thus, we can think of $\frac{1+2\tau}{3}$ and $\frac{2+\tau}{3}$ as (liftings of) the second and the third markings of the domain $\PT$, respectively. See the shaded region in the left side of Figure \ref{fig:input_typeP1333}.

Since $\lambda \in \Z [\tau]$, the images $\lambda \cdot \left(\frac{1+2\tau}{3}\right)$ and $\lambda \cdot \left(\frac{2+\tau}{3} \right)$ will lie in the lattice $\Z \langle \frac{1+2\tau}{3}, \frac{2+\tau}{3} \rangle$ in the universal cover of the target $\PT$. It is clear that the types of these two lattice points determine the orbi-insertions of the orbi-sphere associated with $\lambda$, i.e., if $\lambda \cdot \left(\frac{1+2\tau}{3}\right)$ lies in $\frac{1+2\tau}{3} + \Z \langle 1, \tau \rangle$, then the second orbi-insertion is $\Delta^{1/3}_{2}$ and so on. See Figure \ref{fig:input_typeP1333}.
(Note that $z \mapsto \lambda z$ always send the origin to the origin, which is related to the fact that we fix the first orbi-insertion as $\da{1}{3}$ using the symmetry.)

\begin{figure}
\begin{center}
\includegraphics[height=2in]{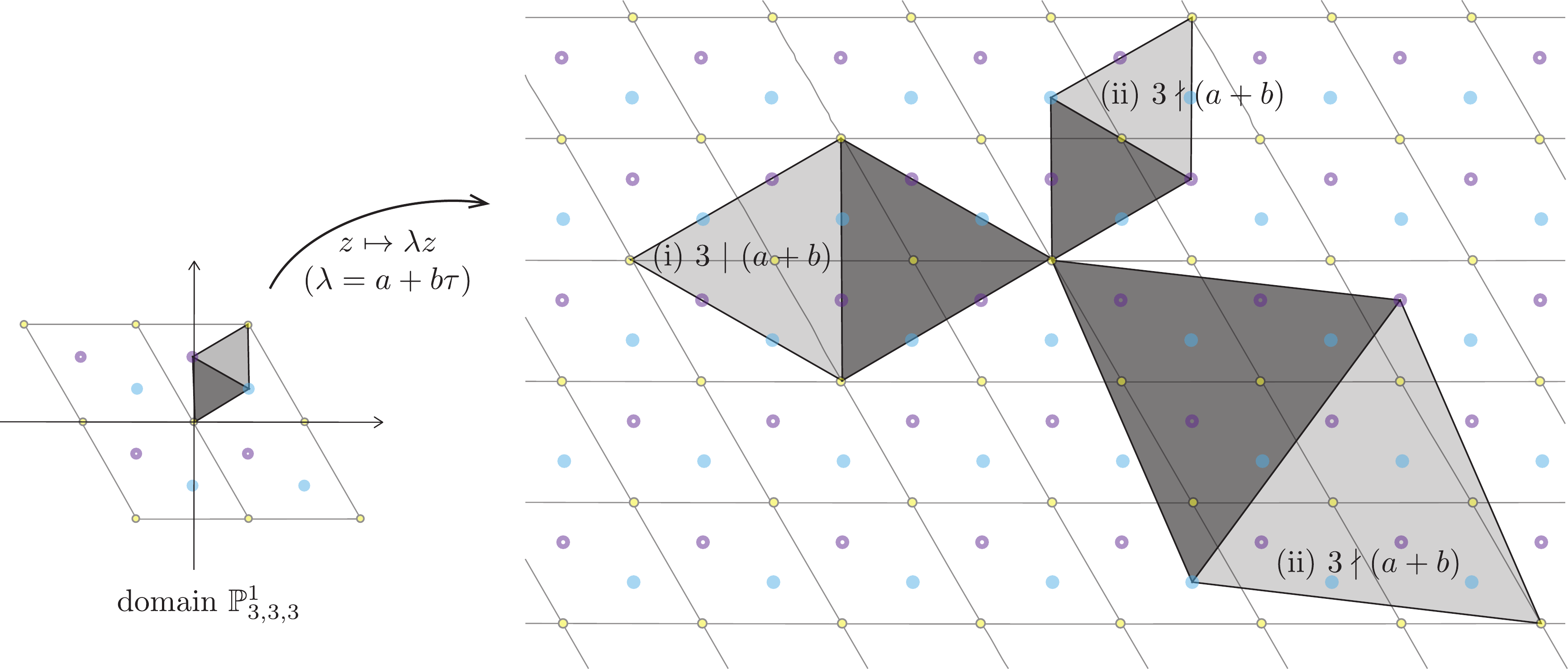}
\caption{Images of holomorphic orbi-spheres in $\PT$ visualized in its universal cover}\label{fig:input_typeP1333}
\end{center}
\end{figure}

Observe that
\begin{eqnarray*}
\lambda \cdot \left( \frac{1+2\tau}{3} \right) &=& (a+ b\tau) \left( \frac{1+2\tau}{3} \right) \\
&=&  \frac{(a-2b) + (2a -b) \tau}{3}
\end{eqnarray*}
and
$$ \lambda \cdot \left( \frac{2+\tau}{3} \right) = \frac{(2a-b) + (a +b)  \tau}{3}.$$
(Here, we used the relation $\tau^2 = -\tau -1$.) 
Using
$$a-2b \equiv a+b \mod 3,$$
we see that there are only two possibilities:
\begin{enumerate}
\item[(i)] $3\,\,|\,\,(a+b)$, for which both $ \lambda \cdot \left( \frac{1+2\tau}{3} \right)$ and $\lambda \cdot \left( \frac{2+\tau}{3} \right)$ correspond to the insertion $\Delta^{1/3}_1$;
\item[(ii)] $3 \nmid (a+b)$, for which both $ \lambda \cdot \left( \frac{1+2\tau}{3} \right)$ and $\lambda \cdot \left( \frac{2+\tau}{3} \right)$ correspond to two different insertions $\Delta^{1/3}_2$ and $\Delta^{1/3}_3$;
\end{enumerate}
We remark that the in case (ii), insertions $(\Delta^{1/3}_2, \Delta^{1/3}_3)$ can be located at either $(z_2,z_3)$ or $(z_3,z_2)$. See the discussion at the end of Section \ref{subsec:symm333}.

Note that 
\begin{equation}\label{eq:dmod3}
d=|\lambda|^2 = a^2 -ab + b^2 = (a+b)^2 -3ab \equiv (a+b)^2 \mod 3.
\end{equation}
Therefore, if $d\equiv 0 \mod 3$, then the corresponding holomorphic spheres contribute to
$$\langle \da{1}{3}, \da{1}{3}, \da{1}{3} \rangle_{0,3},$$
and if $d \equiv 1 \mod 3$, they contribute to
$$\langle \da{1}{3}, \db{1}{3}, \dc{1}{3} \rangle_{0,3}.$$

Let $F(q)$ denote the power series
\begin{equation}\label{eq:F}
F(q):= \displaystyle\sum_{a,b \in \Z} q^{a^2 - ab + b^2}
\end{equation}
See \eqref{eq:FandG} for first few terms of $F$.

By \eqref{eq:dmod3}, the power of any nontrivial term in $F(q)$ should be either $0 \, ({\rm mod} \,\,3)$ or $1\, ({\rm mod} \,\,3)$. Thus, we can decompose $F$ as $F= F_{0,3} + F_{1,3}$ according to the remainder of the power of $q$ by $3$. Then the above discussion directly implies that
$$f_0(q) = \sum_{d \in H_2 ({X}, \Z)} \langle \da{1}{3}, \db{1}{3}, \dc{1}{3} \rangle^{{X}}_{0,3,d} q^d= \frac{1}{6} F_{1,3} (q)$$ 
since there can not be contributions from constant maps.
Here, $\frac{1}{6}$ is responsible for the group $\Z_6 \cong \langle 1+\tau \rangle$ which is discussed at the end of Subsection \ref{subsec:symm333}.

For $f_1(q) =\sum_{d \in H_2 ({X}, \Z)} \langle \da{1}{3}, \da{1}{3}, \da{1}{3} \rangle^{{X}}_{0,3,d} q^d$, there is an additional contribution from the constant map (see Subsection \ref{subsec:contriconst}, \eqref{eq:i7_pt2}) so that $f_1 (q) = \frac{1}{3} F_{0,3} (q)$, where $\frac{1}{3}$ again comes from $\Z_3$, the isotropy group at $w_1$ (or the origin in $\C$).

\begin{remark}
Number theoretic aspects of $F$ such as an explicit description of its Fourier coefficients will be given in Appendix. In particular, we will describe the Fourier coefficients of $F$ in terms of the prime factorization of the exponent of $q$.
\end{remark}

%

\subsection{Contribution from constant maps}\label{subsec:contriconst}
The constant map whose image lies in a single singular point also contributes to the quantum product. Indeed, these constant maps induce the product structure ``$\,\cdot\,$" of the Chen-Ruan cohomology ring \cite{CR1} of $\PT$, and the quantum product deforms this structure analogously to the relation between cup products and quantum products for smooth symplectic manifolds. 

Let us consider one of singular points $w_i$ and constant maps from an orbi-sphere with three markings onto this point. The computation is essentially the same for all $i=1,2,3$ because of the  symmetry. Obviously, there are two constant maps with image $w_i$ whose domain orbi-spheres are $\mathbb{P}^1_{3,3,3}$ and $\mathbb{P}^1_{3,3}$. We denote these maps by $c_1$ and $c_2$. Here, the markings for $c_2$ are located at two singular points and a chosen smooth point. We remark that the second map does not violate Lemma \ref{lem:nodal} since it only holds for non-constant holomorphic orbi-spheres.

$c_1$ and $c_2$ give rise to classical parts 
$$\langle \Delta^{1/3}_i,\Delta^{1/3}_i,\Delta^{1/3}_i \rangle^{{X}}_{0,3,d=0} \quad\mbox{and} \quad \langle \Delta^{1/3}_i, \Delta^{2/3}_i, 1 \rangle^{{X}}_{0,3,d=0}$$ 
of the $3$-fold Gromov-Witten invariant on ${X}=\PT$. Both of these numbers are $\frac{1}{3}$, where the fraction comes from the definition of the orbifold integration \cite{ALR} (see Section \ref{subsec:GWforP}). Therefore, the Chen-Ruan cup product for $\PT$ is given as follows.
\begin{equation}\label{eq:i7_pt}
\Delta^{1/3}_i \cdot \Delta^{2/3}_i = \frac{1}{3} PD(1) = \frac{1}{3} [pt]
\end{equation}
\begin{equation}\label{eq:i7_pt2}
\Delta^{1/3}_i \cdot \Delta^{1/3}_i =  \frac{1}{3} PD(\Delta^{1/3}_i) = \Delta^{2/3}_i
\end{equation}

Here, we used $PD(\Delta^{1/3}_i) = 3 \times \Delta^{2/3}_i$ (Remark \ref{rmk:warnPD}).
\eqref{eq:i7_pt2} completes the computation of $f_1$.

\begin{remark}
In fact, to verify \eqref{eq:i7_pt}, it remains to show that there are no other contributions than constants. However, we have shown in Lemma \ref{lem:nodal} that there are no holomorphic orbi-spheres in $\PT$ which has only two orbifold markings.
\end{remark}

\section{Further applications : (2,3,6), (2,4,4)}\label{sec:236244}

In this section, we proves Theorem \ref{thm:main2} and Proposition \ref{prop:main}.
We slightly modify the classification of holomorphic orbi-spheres in $\PT$ in order to compute the quantum cohomology rings of two other orbifold projective lines with three singular point: $\PH$ and $\PQ$. For a certain product in $QH^\ast_{orb} (\PH)$, we use a heuristic argument, so the proof is incomplete (see Conjecture \ref{conj:P1366}). We hope to fill in the missing part by classifing holomorphic orbi-spheres whose domain admits a hyperbolic structure, and leave it to future investigation. 

We remark that the explicit formulae of quantum products of $\PH$ and $\PQ$ have not appeared in any literature.

\subsection{The product on $QH^\ast_{orb} (\PH)$}
We set the notation for generators of $H^*_{orb}(\PH)$ as follows.
Recall that $E$ is the elliptic curve associated with the lattice $ \Z  \langle 1, \tau \rangle$ in $\C$ where $\tau =\exp \left(\frac{2 \pi \sqrt{-1}}{3} \right)$. Then $\PH$ is obtained as the global quotient $[ E / \Z_6 ]$, where $\Z_6 \cong\langle 1+\tau \rangle$ acts on $E$ by the complex multiplication. There are three cone points on $\PH$ and we use the same notation $w_1, w_2$, and $w_3$ for these singular point as we did for $\PT$, where $w_1$, $w_2$ and $w_3$ have isotropy groups $\Z_2, \Z_3$, and $\Z_6$, respectively. 
The inertia orbifold $\mathcal{I}\PH$ consists of the smooth sector, $B\Z_2$, $B\Z_3$, and $B\Z_6$.
The $\Q$-basis of $H^*_{orb}(\PH, \Q)$ is given as $1, \da{1}{2}, \db{1}{3}, \db{2}{3}, \dc{1}{6}, \cdots, \dc{5}{6}, \pt$ as follows.

The basis of smooth sector are
\begin{align*}
&H^0_{orb}(\PH, \Q) = \Q \cdot 1, \quad \quad \quad H^2_{orb}(\PH, \Q) = \Q \cdot \pt.
\end{align*}
For twist sectors, let $\da{1}{2} \in H^1_{orb}(\PH, \Q)$, $\db{j}{3} \in H^{\frac{2j}{3}}_{orb}(\PH, \Q)$($j=1,2$), and $\dc{k}{6} \in H^{\frac{2k}{6}}_{orb}(\PH, \Q)$($k=1,\cdots,5$) which are supported at singular points $w_1, w_2$, and $w_3$, respectively.
From the virtual dimension formula of $\OL{\CM}_{0,3,d}(\PH)$, we can classify all possible orbi-insertions with expected dimension $0$ and the corresponding domain orbi-sphere as in the following list.
\begin{enumerate}[(a)]
	\item $\PH$ : $\langle \da{1}{2}, \db{1}{3}, \dc{1}{6} \rangle$, $\langle \dc{3}{6}, \db{1}{3}, \dc{1}{6} \rangle$ $\langle \da{1}{2}, \dc{2}{6}, \dc{1}{6} \rangle$, $\langle \dc{3}{6}, \dc{2}{6}, \dc{1}{6} \rangle$,
	\item $\PT$ : $\langle \dc{2}{6}, \dc{2}{6}, \dc{2}{6} \rangle$, $\langle \db{1}{3}, \dc{2}{6}, \dc{2}{6} \rangle$, 	 $\langle \db{1}{3}, \db{1}{3}, \dc{2}{6} \rangle$, $\langle \db{1}{3}, \db{1}{3}, \db{1}{3} \rangle$ 
	\item $\PP^1_{3,6,6}$(hyperbolic) : $\langle \dc{1}{6}, \dc{1}{6}, \dc{4}{6} \rangle$, $\langle \db{2}{3}, \dc{1}{6}, \dc{1}{6} \rangle$	
	\item $\PP^1_{2,2}$ : $\langle 1, \dc{3}{6}, \dc{3}{6} \rangle$, $\langle 1, \da{1}{2}, \dc{3}{6} \rangle$, $\langle 1, \da{1}{2}, \da{1}{2} \rangle$
	\item $\PP^1_{3,3}$ : $\langle 1, \db{1}{3}, \dc{4}{6} \rangle$, $\langle 1, \db{2}{3}, \dc{2}{6} \rangle$
\end{enumerate}
From Lemma \ref{lem:nodal}, there are no nontrivial maps which contribute to the type of (4) and (5). Thus, if we denote $\bt := \sum t_{j,i} \Delta^i_j$, the genus 0-Gromov-Witten potential of $\PH$ can be written up to order of $t^3$ as follows:
\begin{equation}
\begin{array}{rl}
F^{\PH}_{0} (\bt) =& \frac{1}{2} t_0^2 \log q + t_{0} (\frac{1}{2} \ta{1}{2} \ta{1}{2} + \frac{1}{6} \tc{3}{6} \tc{3}{6}) + 
+ (\ta{1}{2} \tb{1}{3} \tc{1}{6}) \cdot h_0 (q) \\
&+ (\tc{3}{6} \tb{1}{3} \tc{1}{6}) \cdot h_1 (q) + (\ta{1}{2} \tc{2}{6} \tc{1}{6}) \cdot h_2 (q) + (\tc{3}{6} \tc{2}{6} \tc{1}{6}) \cdot h_3 (q)\\
&+ \frac{1}{6} \tc{2}{6}^3 \cdot h_4 (q) + \frac{1}{2} \tc{2}{6}^2 \tb{1}{3} \cdot h_5 (q) + \frac{1}{2} \tc{2}{6} \tb{1}{3}^2 \cdot h_6 (q) + \frac{1}{6} \tb{1}{3}^3 \cdot h_7 (q)\\
&+ \frac{1}{2} \tc{1}{6}^2 \tc{4}{6} \cdot h_8 (q) + \frac{1}{2} \tc{1}{6}^2 \tb{2}{3} \cdot h_9 (q) + \frac{1}{2} \tb{2}{3} \tc{1}{6}^2 \cdot h_{10} (q) + O(t^4),
\end{array}
\end{equation}
where the precise expressions of $h_i(q)$ for $0 \leq i \leq 10$ will be given, later.

For holomorphic orbi-spheres of type (a) and (b), we choose the presentations of domain orbi-spheres as $[E / \Z_6]$ and $[E / \Z_3]$, respectively. Here, $E$ is the elliptic curve corresponding to the $\Z$-lattice $\langle 1, \tau \rangle$ in $\C$, where $\tau = \exp{\frac{2\pi\sqrt{-1}}{3}}$.

Observe that any holomorphic orbi-sphere with the orbi-insertion condition in (a) or (b) satisfies the condition in the lemma \ref{lem:u_orb_cover}. So, we can lift such maps $u :  \PH \to \PH$ and $u : \PT \to \PH$ to a linear map between universal coverings $\WT{u} : \C \to \C$.
Below, we will count these holomorphic orbi-spheres with help of the lattice structures of inverse image of orbi-singular points in $\C$.
As in the case of $\PT$, it will turn out that the counting matches  the number of solutions of a certain Diophantine equations.
The regularity of these holomorphic orbi-spheres are guaranteed by Lemma \ref{lem:reg}.

To clarify the orbi-insertions by looking at the lifted linear map $\tilde{u} : \C \to \C$, we explicitly identify the lattice structure on $\C$ coming from the universal orbifold covering map $p: \C \to \PH$ as follows:
\begin{equation}\label{eq:lattice_PH}
\begin{array}{rl}
p^{-1} (w_1) &= \Z \left\langle \frac{1}{2}, \frac{\tau}{2} \right\rangle + p^{-1}(w_3) \\
& = \left\{ \frac{a}{2} + \frac{b}{2} \tau \mid a, b \in \Z \text{ and $a$ or $b$ is an odd integer} \right\},	\\
p^{-1} (w_2) &= \Z \left\langle \frac{2+\tau}{3}, \frac{1+ 2\tau}{3} \right\rangle + p^{-1}(w_1),\\
p^{-1} (w_3) &= \Z \langle 1, \tau \rangle \,\,(\ni 0).

\end{array}
\end{equation}
In particular, $w_3$ is set to be a base point associated with the universal covering $(\C,0)$. (See Figure \ref{fig:236lattice})
\begin{figure}
\begin{center}
\includegraphics[height=3.8in]{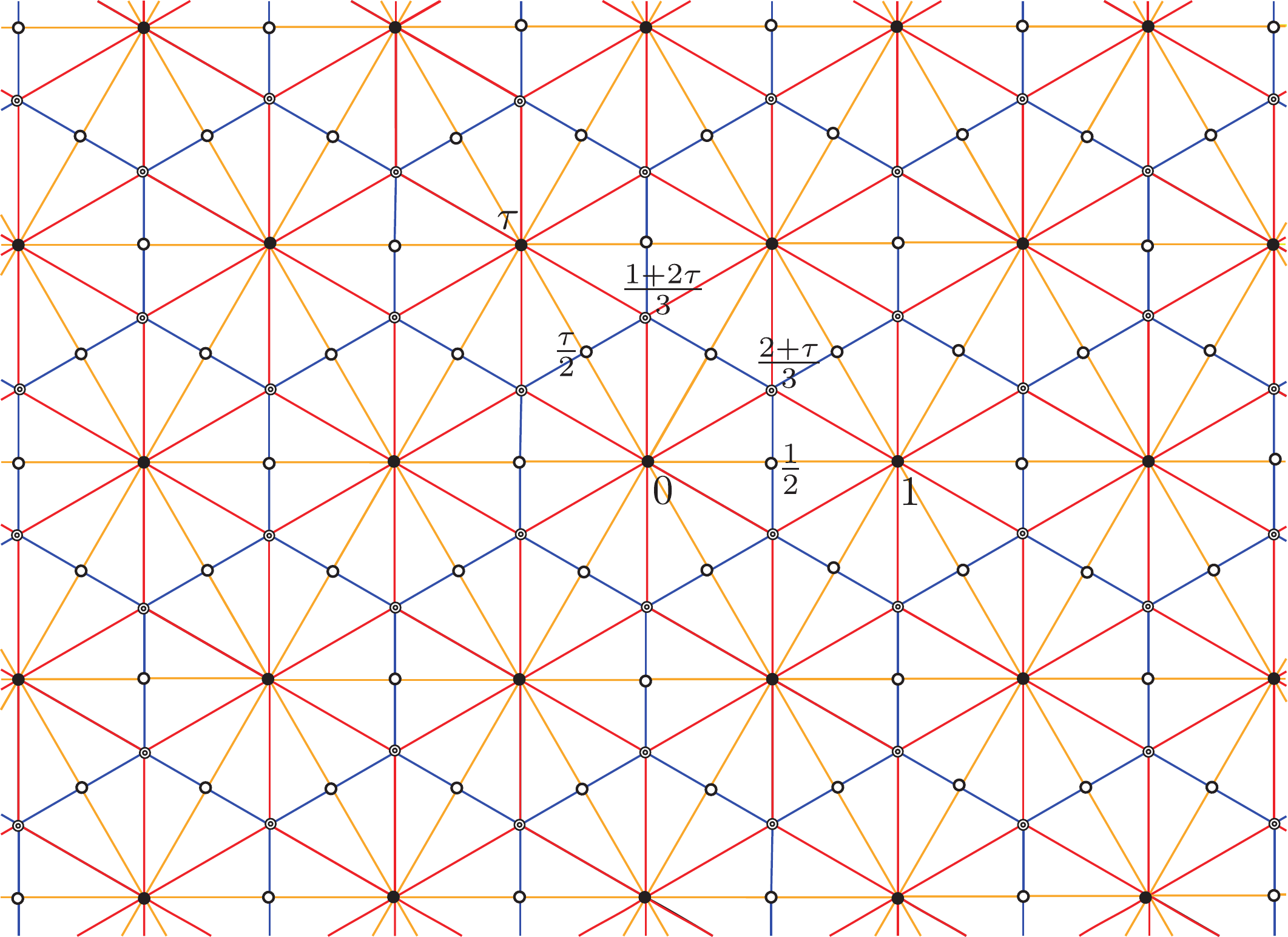}
\caption{Lattices on the universal cover of $\PH$ : $p^{-1} (w_1)=\{ \circ \}$, $p^{-1}(w_2)=$ $\{${\tiny $\circledcirc$}$\}$ and $p^{-1} (w_3)= \{\bullet\}$}\label{fig:236lattice}
\end{center}
\end{figure}

The universal cover $\C$ of the domain $\PH$ also has the same lattice structure, and the lattices on $\C$ from the domain $\PT$ are given as in Section \ref{subsec:compP1333}.

\subsubsection*{Case {\rm (a)} with the domain orbi-sphere $\PH$ : $h_i$ for $0 \leq i \leq 3$}
Let $z_1$, $z_2$, and $z_3$ be the three orbi-points in the domain $\PH$, whose orders of singularities are 2, 3, and 6, respectively.
If $u$ is a holomorphic map from $\PH$ to itself with the orbi-insertion condition as in (a), then $u$ is a orbifold covering map by Lemma \ref{lem:u_orb_cover}, so one can find the lifting $\WT{u} : \C \to \C$ with $\WT{u}(0) = 0$:
\begin{equation}
	\xymatrix{
		\C \ar@{.>}[r]^{\WT{u}} \ar[d]^{p} & \C \ar[d]^{p} \\
		\PH \ar[r]^{u} & \PH.
		}
\end{equation}
Since any holomorphic orbi-sphere $u$ contributing to (a) maps $z_3$ to $w_3$ (by the arrangement of insertions in (a)), $\WT{u}(z) = \lambda z$ for some $\lambda \in \Z[\tau]$. Conversely, it is clear from Figure \ref{fig:236lattice} that any such linear map $\WT{u}$ descends to a holomorphic orbi-sphere with insertions as in (a).
Since for $\lambda = a + b \tau$ ($a, b \in \Z$), the degree of the underlying map of $\WT{u}(z) = \lambda z$ is $N := |\lambda|^2= \lambda \OL{\lambda} = a^2 -ab + b^2$, the above discussion shows that 
$$h_0 (q) + h_1 (q) + h_2 (q) + h_3 (q) = \frac{1}{6} F(q),$$
where $F(q)$ is defined by the equation \eqref{eq:F}.
Here, $\frac{1}{6}$ in the right hand side comes from the symmetry between linear maps which induce the same holomorphic orbi-sphere.
By the same argument as in Section \ref{subsec:symm333},  we see that the symmetry among these linear maps is generated by the $(1+\tau)$-multiplication, which is nothing but the action of the isotropy group of $w_3$ (isomorphic to $\Z_6$). 

\begin{figure}
\begin{center}
\includegraphics[height=1.7in]{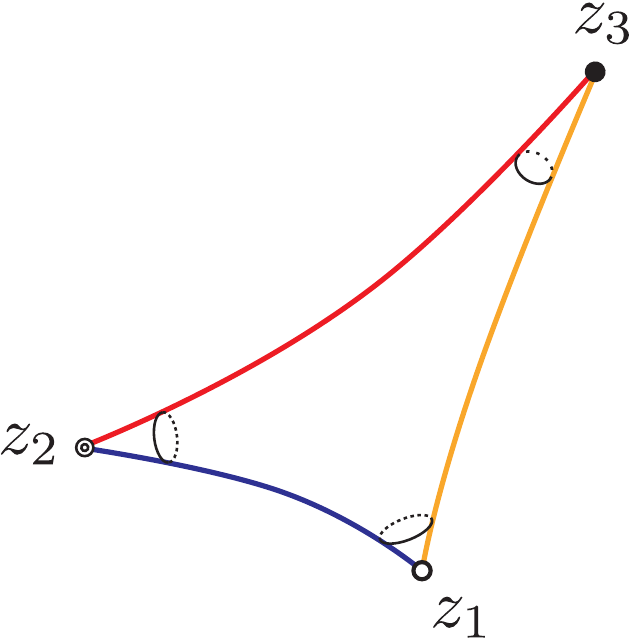}
\caption{$\PH$}\label{fig:236threed}
\end{center}
\end{figure}


Note that the triangle whose vertices are $0\in p^{-1}(z_3)$, $\frac{2 + \tau}{3} \in p^{-1}(z_2)$, and $\frac{1 + \tau}{2} \in p^{-1}(z_1)$ gives the fundamental domain of the upper-hemisphere of (the domain) $\PH$ (see the shaded region in Figure \ref{fig:236lattice} and compare it with Figure \ref{fig:236threed}).
As in the case of $\PT$, we classify the orbi-insertion condition by chasing the images of $\frac{2 + \tau}{3}$ and $\frac{1 + \tau}{2}$ in the domain.
For $\lambda = a + b\tau$,
$$\lambda \cdot \frac{2 + \tau}{3} = \frac{2a-b}{3} + \frac{a + b}{3} \tau$$
and
$$\lambda \cdot \frac{1 + \tau}{2} = \frac{a-b}{2} + \frac{a}{2} \tau.$$
First, note that
$$\lambda \cdot \frac{1 + \tau}{2} \in p^{-1} (w_1) \quad \iff\quad \text{$a$ or $b$ is odd},$$
$$\lambda \cdot \frac{1 + \tau}{2} \in p^{-1} (w_3) \quad \iff\quad \text{$a$ and $b$ is even}.$$
From $( 3 \mid 2a - b \iff 3 \mid a + b)$, it can be easily checked that 
$$\lambda \cdot \frac{2 + \tau}{3} \in p^{-1}(w_3) \quad \iff\quad 3 \mid (a + b),$$
$$\lambda \cdot \frac{2 + \tau}{3} \in p^{-1}(w_2) \quad \iff\quad 3 \nmid (a + b).$$
Hence using the equation \eqref{eq:dmod3}, there are two possible orbi-insertions at the marked point corresponding to $\lambda \cdot \frac{2 + \tau}{3}$ : 
\begin{description} 
\item[$\dc{2}{6}$] $3 \mid (a + b) \iff N \equiv 0 \mod 3$,
\item[$\db{1}{3}$] $3 \nmid (a + b) \iff N \equiv 1 \mod 3$.
\end{description}
Similarly, two possible orbi-insertions at the marked point corresponding to $\lambda \cdot \frac{1 + \tau}{2}$ are
\begin{description}
\item[$\dc{3}{6}$] $a$ and $b$ is even $\iff N \equiv 0 \mod 2$,
\item[$\da{1}{2}$] $a$ or $b$ is odd \quad $\iff N \equiv 1 \mod 2$.
\end{description}
Summarizing the above discussion, we conclude from the Chinese remainder theorem that $u$ contributes to

\begin{itemize}
\item[]  $\langle \da{1}{2}, \db{1}{3}, \dc{1}{6} \rangle$ $\iff {\rm deg}\, u \equiv 1 \mod 6$,
\item[] $\langle \dc{3}{6}, \db{1}{3}, \dc{1}{6} \rangle$ $\iff {\rm deg}\, u \equiv 4 \mod 6$.
\item[]  $\langle \da{1}{2}, \dc{2}{6}, \dc{1}{6} \rangle$ $\iff {\rm deg}\, u \equiv 3 \mod 6$,
\item[]  $\langle \dc{3}{6}, \dc{2}{6}, \dc{1}{6} \rangle$ $\iff {\rm deg}\, u \equiv 0 \mod 6$.
\end{itemize}

Recall $N=|\lambda|^2 = {\rm deg}\, u$, which equals the exponent of $q$ for the term in the Gromov-Witten potential that $u$ contributes to. Therefore, we obtain

\begin{align*}
h_0 (q) &=  \frac{1}{6} \sum_{\substack{N =1 \\ N \equiv 1 \text{ mod } 6}}^{\infty} \sum_{\substack{m^2 - mn + n^2 = N \\ m,n\in \Z}} q^N = \frac{1}{6} F_{1,6}(q),\\
h_1 (q) &=  \frac{1}{6} \sum_{\substack{N = 1 \\ N \equiv 4 \text{ mod } 6}}^{\infty} \sum_{\substack{m^2 - mn + n^2 = N \\ m,n\in \Z}} q^N = \frac{1}{6} F_{4,6}(q),\\
h_2 (q) &=  \frac{1}{6} \sum_{\substack{N = 1 \\ N \equiv 3 \text{ mod } 6}}^{\infty} \sum_{\substack{m^2 - mn + n^2 = N \\ m,n\in \Z}} q^N = \frac{1}{6} F_{3,6}(q),\\
h_3 (q) &= \frac{1}{6} \left( 1 + \sum_{\substack{N = 1 \\ N \equiv 0 \text{ mod } 6}}^{\infty} \sum_{\substack{m^2 - mn + n^2 = N \\ m,n\in \Z}} q^N \right) = \frac{1}{6} + \frac{1}{6} F_{0,6}(q).
\end{align*}
where $F_{i,6}$ is the sum of terms in $F$ whose exponents of $q$ is $i$ modulo $6$.
Here, the constant term of $h_3$ can be obtained from a similar argument in the subsection \ref{subsec:contriconst}.

\subsubsection*{Case {\rm (b)} with the domain orbi-sphere $\PT$}
We first show that holomorphic orbi-spheres with orbi-insertions as in case (b) can be lifted to the one on $\PT$.
Let $\pi : \PT \to \PH$ be the 2-fold orbifold covering map which comes from the action on $[E/ \langle \tau \rangle]$ generated by the $(1+\tau)$-multiplication, as drawn in Figure \ref{fig:2fold333}. Write $w_i'$ and $w_i$ ($i=1,2,3$) for orbi-points in $\PT$ and $\PH$, respectively and let $\pi$ send both $w_1'$ and $w_2'$ to $w_2$, and $w_3'$ to $w_3$.

\begin{figure}
\begin{center}
\includegraphics[height=1.5in]{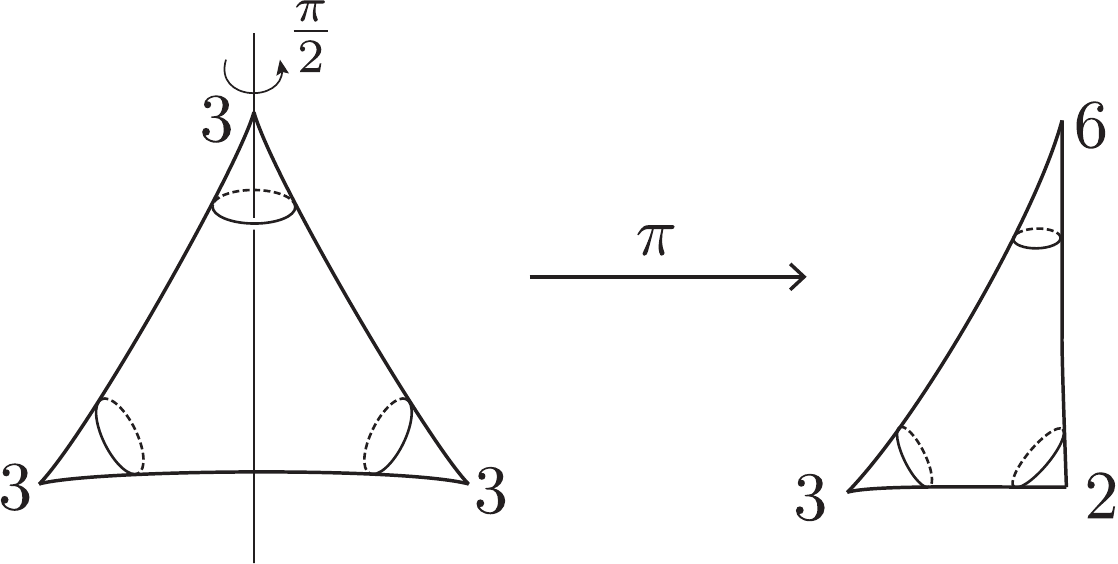}
\caption{The $2$-fold covering $\pi : \PT \to \PH$}\label{fig:2fold333}
\end{center}
\end{figure}

After fixing base points of $\PT$ and $\PH$,
\begin{align*}
\pi_1^{orb} (\PT) &= \left\langle \rho_1, \rho_2, \rho_3 \mid (\rho_1)^3 = (\rho_2)^3 = (\rho_3)^3 = \rho_1 \rho_2 \rho_3 = 1 \right\rangle,\\
\pi_1^{orb} (\PH) &= \left\langle \lambda_1, \lambda_2, \lambda_3 \mid (\lambda_1)^2 = (\lambda_2)^3 = (\lambda_3)^6 = \lambda_1 \lambda_2 \lambda_3 = 1 \right\rangle.
\end{align*}
(see Section \ref{subsec:orbfundgp}) and $\pi$ induces a group homomorphism $\pi_\ast : \pi_1^{orb} (\PT) \to \pi_1^{orb} (\PH)$.
We see from Figure \ref{fig:2fold333} that the images of $\rho_1$ and $\rho_2$ under $\pi_\ast$ lie in the conjugacy class of $\lambda_2$, and the image of the other generator $\rho_3$ lies in that of $(\lambda_3)^2$. (Here, conjugacy classes depend on the choice of base points.)
It follows that $\pi_* \left(\pi_1^{orb} (\PT)\right)$ contains $\lambda_2$ and $(\lambda_3)^2$, as it is a normal subgroup of $\pi_1^{orb} (\PH)$.

\begin{lemma}\label{lem:PT_lifting}
For a given (b)-type holomorphic orbi-sphere $u : \PT \to \PH$, there exists a holomorphic orbi-sphere $\WT{u}$ which make the following diagram commute:
\begin{equation*}
\xymatrix{
{} & \PT \ar[d]^{\pi} \\
\PT \ar@{.>}[ur]^{\exists \WT{u}} \ar[r]^{u} & \PH.
}
\end{equation*}
\end{lemma}
\begin{proof}
Observe that only two kinds of orbi-insertions $\db{1}{3}$ and $\dc{2}{6}$ appear in (b). 
Hence $u_*$ maps a generator of $\pi_1^{orb} (\PT)$ to an element in the conjugacy class of $\lambda_2$ or $\lambda_3^2$. (Indeed, if we choose base points and the generators $\rho_1$ and $\lambda_2$ as in (b) of Figure \ref{fig:orbfundgenpath}, then $u_*$ sends $\rho_1$ exactly to $\lambda_2$, and similar happens for $\rho_2$ and $\rho_3$.)
Thus, $u_* \left(\pi_1^{orb} (\PT)\right)$ is contained in $\pi_* \left(\pi_1^{orb} (\PT)\right)$.
From Proposition \ref{prop:cov_lifting}, there exists an orbi-map $\WT{u} : \PT \to \PT$ which lifts $u$.
\end{proof}


Let $z$ be the one of three orbi-points of the domain $\PT$ for a holomorphic orbi-sphere $u$ of type (b).
If $u$ sends $z$ to $w_2 \in \PH$ with the orbi-insertion $\db{1}{3}$, then the corresponding orbi-insertion of the lifting $\WT{u} : \PT \to \PT$ is $\da{1}{3}$ or $\db{1}{3}$.
Similarly, if $u(z) = w_3$ with the insertion $\dc{2}{6}$, then the corresponding orbi-insertion of a lifting $\WT{u}$ is $\dc{1}{3}$.
Here, we abused the notation for orbi-insertions of $\PT$ and $\PH$.

For each holomorphic orbi-sphere $u : \PT \to \PH$ with orbi-insertion of type (b), there are two liftings $\tilde{u} : \PT \to \PT$.
Two liftings of $u$ are related by the $\Z_2$-action (i.e. the action of the deck transformation group) which switches $w_1'$ and $w_2'$. Therefore, if one lifting has orbi-insertion $\da{1}{3}$, then the other lifting has orbi-insertion $\db{1}{3}$. 

In summary, Lemma \ref{lem:PT_lifting} gives rise to the following one-to-two correspondences:
\begin{equation*}
\begin{array}{rcl}
\langle \dc{2}{6}, \dc{2}{6}, \dc{2}{6} \rangle^{\PH}&\stackrel{1:2}{\longleftrightarrow}&\langle \dc{1}{3}, \dc{1}{3}, \dc{1}{3} \rangle^{\PT} \\
\langle \db{1}{3}, \db{1}{3}, \dc{2}{6} \rangle^{\PH} &\stackrel{1:2}{\longleftrightarrow}& \langle \da{1}{3}, \db{1}{3}, \dc{1}{3} \rangle^{\PT} + \langle \db{1}{3}, \da{1}{3}, \dc{1}{3} \rangle^{\PT} \\
&& = 2\langle \da{1}{3}, \db{1}{3}, \dc{1}{3} \rangle^{\PT}  \\
\langle \db{1}{3}, \db{1}{3}, \db{1}{3} \rangle^{\PH} &\stackrel{1:2}{\longleftrightarrow}& \langle \da{1}{3}, \da{1}{3}, \da{1}{3} \rangle^{\PT} + \langle \db{1}{3}, \db{1}{3}, \db{1}{3} \rangle^{\PT} \\
&& = 2 \langle \da{1}{3}, \da{1}{3}, \da{1}{3} \rangle^{\PT}.
\end{array}
\end{equation*}
($\langle \db{1}{3}, \dc{2}{6}, \dc{2}{6} \rangle^{\PH}$ vanishes since there are no corresponding liftings.)

Therefore, $h_i$ for $4 \leq i \leq 7$ is given as follows:

\begin{prop}
Let $f_0^{\PT}(q)$ and $f_1^{\PT}(q)$ be the coefficient of the $t_1 t_2 t_3$ and $t_i^3$ of $F_0^{\PT}$, respectively. Then
	\begin{align*}
		h_4 (q) &= \frac{1}{2} f_1^{\PT} (q^2)
				=  \frac{1}{6} + q^6 + q^{18} + q^{24} + 2 q^{42} + O(q^{48}),	\\
		h_5 (q) &= 0,	\\
		h_6 (q) &= f_0^{\PT} (q^2) = q^2 + q^8 + 2 q^{14} + 2 q^{26} + q^{32} + 2 q^{38} + O(q^{48}),\\
		h_7 (q) &= f_1^{\PT} (q^2)
				= \frac{1}{3} + 2 q^6 + 2 q^{18} + 2 q^{24} + 4 q^{42} + O(q^{48}).
	\end{align*}
\end{prop}

\subsubsection*{Case {\rm (c)} with the domain orbi-sphere $\PP^1_{3,6,6}$}
For these kind of contributions, the lifting of holomorphic orbi-spheres on the universal cover level is no longer a linear map, since the domain orbi-sphere is hyperbolic. Hence we can not use our classification argument any more. However, we may try to find such maps directly by looking at their image on the universal cover $\C$ of the target $\PH$.

For this, we consider rhombi in the universal covering of $\PH$ whose vertices lie in the $p^{-1}(w_1, w_2, w_3)$. For example, observe that the rhombus $v$ whose set of vertices are $\left\{0, \frac{2+\tau}{3}, 1, \frac{1-\tau}{3} \right\}$ gives one of contribution from $\PP^1_{3,6,6}$ to $\langle \Delta^{2/3}_2, \Delta^{1/6}_3,\Delta_2^{1/6} \rangle$. (See the rightmost Rhombus in Figure \ref{fig:366conj}.) One can visualize this holomorphic orbi-sphere by folding this rhombus along its longer diagonal. Pairs of identified edges after this process are drawn in Figure \ref{fig:366conj}.

\begin{figure}
\begin{center}
\includegraphics[height=2.4in]{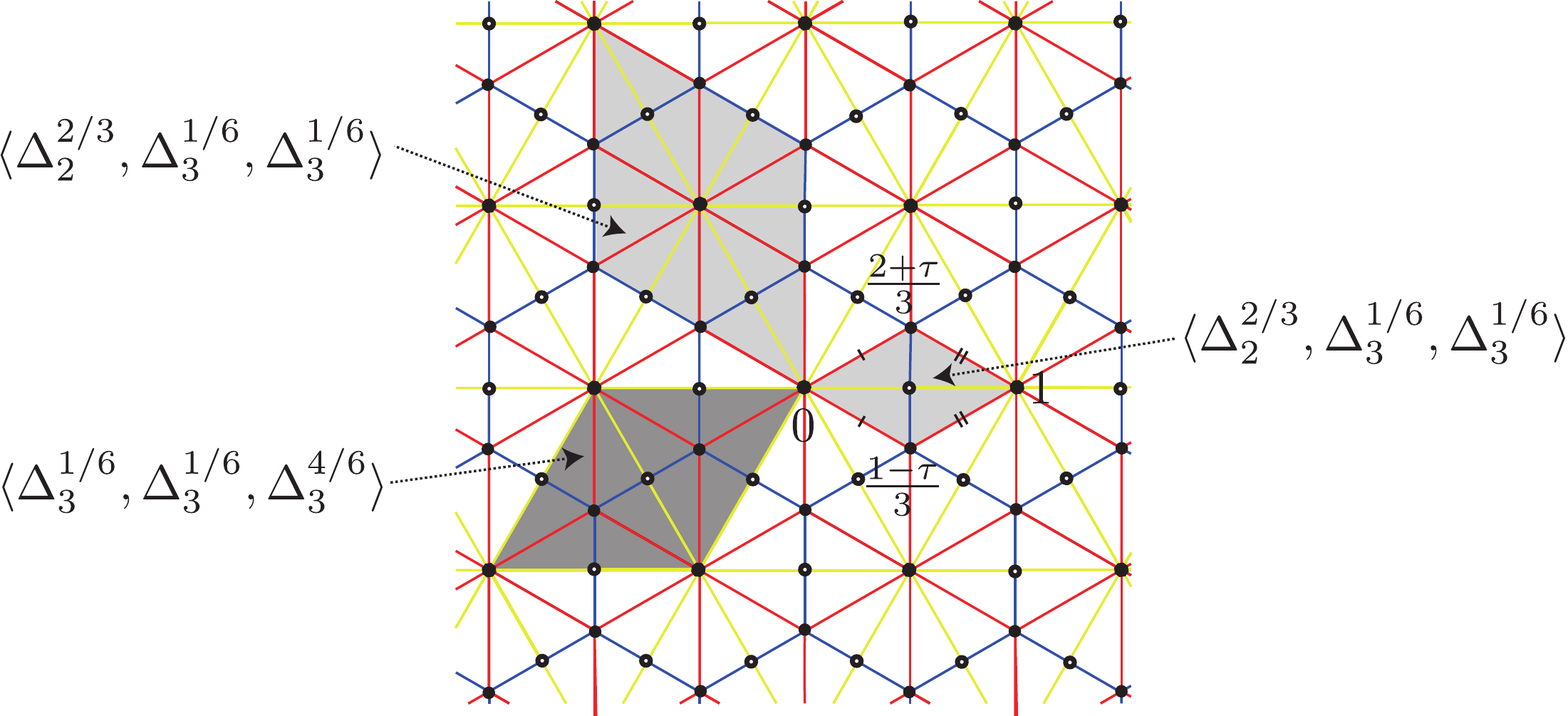}
\caption{Images of holomorphic orbi-spheres $\PP^1_{3,6,6} \to \PH$ visualized in the universal cover of $\PH$}\label{fig:366conj}
\end{center}
\end{figure}

There are various such rhombi, and their corresponding orbi-insertions can be classified in the following way. Note that these rhombi are images of the smallest rhombus $v$ given above by  linear maps $z \mapsto \lambda z$ for $\lambda \in \Z [\tau]$. (This means, we regard the vertices $\frac{2+ \tau}{3}$ and $\frac{1-\tau}{3}$ of $v$ as markings $z_2$ and $z_3$ of the order $6$ in the domain $\mathbb{P}^1_{3,6,6}$, respectively.)
Since $(1+\tau) \cdot \frac{1-\tau}{3} = \frac{2 + \tau}{3}$, insertions at $z_2$ and $z_3$ are the same, and if
$\lambda \cdot \left(\frac{2+\tau}{3}\right)$ is contained in $p^{-1}(w_3)$ (resp. $p^{-1}(w_2)$), the corresponding insertions are $\langle \Delta^{1/6}_3, \Delta^{1/6}_3, \Delta^{4/6}_3 \rangle$ (resp. $\langle \db{2}{3}, \dc{1}{6}, \dc{1}{6} \rangle$). Recall that $h_8$ counts $\langle \Delta^{1/6}_3, \Delta^{1/6}_3, \Delta^{4/6}_3 \rangle$ whereas $h_9$ counts $\langle \db{2}{3}, \dc{1}{6}, \dc{1}{6} \rangle$.

Using the identity $\lambda \cdot \frac{2 + \tau}{3} = \frac{2a-b}{3} + \frac{a + b}{3} \tau$ and proceeding as in case (a), we have
$$\lambda \cdot \frac{2 + \tau}{3} \in p^{-1}(w_3) \iff 3 \mid N (= a^2 - ab + b^2)$$
and
$$\lambda \cdot \frac{2 + \tau}{3} \in p^{-1}(w_2 ) \iff 3 \nmid N (= a^2 - ab + b^2)$$
It is easy to see that six rhombi related by $\Z_6$-rotation at the origin represent the same map, and the degrees of these rhombi are also given by $|\lambda|^2$. Comparing with the decomposition of $F$ in terms of $q$-th power (mod 3) as in Section \ref{subsec:compP1333}, it follows that
$h_8 (q) = \frac{1}{6} F_{0,3}(q^2)$ and $h_9 (q) = \frac{1}{6} F_{1,3}(q^2)$, if one can prove that there are {\em no} other contributions.

\begin{conjecture}\label{conj:P1366}
We conjecture that there are no contributions from $\PP^1_{3,6,6,}$ other that these rhombi, or equivalently,
\begin{align*}
	h_8 (q) &= \frac{1}{6} F_{0,3}(q^2) = \frac{1}{2} f^{\PT}_1 (q^2) \\
			&= \frac{1}{6} + q^6 + q^{18} + q^{24} + 2 q^{42} + q^{54} + q^{72} + 2 q^{78} + O(q^{96}),\\
	h_9 (q) &= \frac{1}{6} F_{1 ,3}(q^2) = f^{\PT}_0 (q^2) \\
			&= q^2 + q^8 + 2 q^{14} + 2 q^{26} + q^{32} + 2 q^{38} + O(q^{48}).
\end{align*}
\end{conjecture}

%

There is a nontrivial algebraic relation between $h_8(q)$ and $h_9(q)$ which basically comes from the Frobenius structure on $QH^\ast_{orb} (\PH;\Q)$. This can be obtained as follows: firstly,
\begin{equation}\label{eq:Frob366}
\begin{array}{rcl}
h_8 &=& \left( \Delta^{1/6}_3 \ast \Delta^{1/6}_3, \Delta^{4/6}_3 \right) \\
&=& \left( \Delta^{1/6}_3 \ast \Delta^{1/6}_3, \frac{1}{6}(h_6)^{-1} \Delta_2^{1/3} \ast \Delta_2^{1/3} -  \frac{1}{2} (h_6)^{-1} h_7 \Delta_2^{2/3} \right) \\
&=& - \frac{1}{2} (h_6)^{-1} h_7 h_9   + \frac{1}{6}(h_6)^{-1}  \left( \Delta^{1/6}_3 \ast \Delta^{1/6}_3,  \Delta_2^{1/3} \ast \Delta_2^{1/3}  \right) 
\end{array}
\end{equation}
where $\left( \,\, , \,\, \right)$ is the Poincar\`e paring and in the second equality, we used 
$$ \Delta_2^{1/3} \ast \Delta_2^{1/3} = 6 h_6 \Delta_{3}^{4/6} + 3 h_7 \Delta_2^{2/3}$$
which is completely known from cases (b) and (c).

The last term in \eqref{eq:Frob366} can be computed with the help of the Frobenius structure:
\begin{eqnarray*}
\left( \Delta^{1/6}_3 \ast \Delta^{1/6}_3, \Delta_2^{1/3} \ast \Delta_2^{1/3}  \right) &=& \left( \Delta^{1/6}_3, \Delta_3^{1/6} \ast \Delta_2^{1/3} \ast \Delta_2^{1/3}  \right) \\
&=& \left( \Delta^{1/6}_3, ( 6 h_1 \Delta^{3/6}_3 + 2 h_0 \Delta^{1/2}_1 ) \ast \Delta_2^{1/3}  \right) \\
&=& 6h_1 \left( \Delta^{1/6}_3, \Delta^{3/6}_3 \ast \Delta_2^{1/3} \right) + 2h_0 \left( \Delta^{1/6}_3, \Delta^{1/2}_1 \ast \Delta_2^{1/3} \right)\\
&=&  6 (h_1)^2 +2 (h_0)^2.
\end{eqnarray*}

Plugging it into \eqref{eq:Frob366}, we obtain the relation
\begin{equation}\label{eq:numcheck}
  6 h_6 h_8 = - 3 h_7 h_9 + 6  (h_1)^2 + 2  (h_0)^2.
\end{equation}

One can check \eqref{eq:numcheck} numerically up to a higher enough order using Mathematica with our conjectural $h_8$ and $h_9$.

\begin{remark}
A similar kind of lifting argument as in case (b) tells us that $(3,6,6)$-contributions for $\PT$ is equivalent to a certain kind of $4$-fold Gromov-Witten of $\PT$ which counts holomorphic orbi-spheres $\PP^1_{3,3,3,3} \to \PT$.
\end{remark}

\subsection{The product on $QH^\ast_{orb} (\PQ)$}
Let $E'$ be the elliptic curve associated with the lattice $\Z \langle 1, i \rangle$, where $i = \sqrt{-1}$. (In fact, $E$ and $E'$ are isomorphic as symplectic manifolds.) Then the quotient of $E'$ by the $\Z_4$-action which is generated by the $i(=\sqrt{-1})$-multiplication is the elliptic orbifold projective line $\PQ = [E' / \Z_4]$ with three singular points $w_1, w_2, w_3$. Here, $w_1$ is the point with the local group isomorphic to $\Z_2$, and $w_2, w_3$ have local groups isomorphic to $\Z_4$.

The inertial orbifold $\mathcal{I}\PQ$ consists of the smooth sector together with a $B\Z_2$ and two $B\Z_4$'s.
As usual, the $\Q$-basis of $H^*_{orb}(\PQ, \Q)$ is taken as 
$$1, \da{1}{2}, \db{1}{4}, \db{2}{4}, \db{3}{4}, \dc{1}{4}, \dc{2}{4}, \dc{3}{4}, \pt$$
Then the cohomology of the smooth sector is given by
\begin{align*}
&H^0_{orb}(\PQ, \Q) = \Q \cdot 1, \quad \quad \quad H^2_{orb}(\PQ, \Q) = \Q \cdot \pt.
\end{align*}
For twist sectors, $\da{1}{2} \in H^1_{orb}(\PQ, \Q)$, $\Delta_k^{j/4} \in H^{\frac{2j}{4}}_{orb}(\PH, \Q)$($j=1,2,3$, $k=2, 3$)  are generators supported at singular points $w_1, w_2, w_3$, respectively.
In a similar way with $\PH$ case, we classify all the triple orbi-insertions with expected dimension $0$ and their domain orbifolds.

\begin{enumerate}[(a)]
	\item $\PQ$ : $\langle \da{1}{2}, \Delta_j^{1/4}, \Delta_k^{1/4} \rangle$, $\langle \Delta_j^{2/4}, \Delta_j^{1/4}, \Delta_k^{1/4} \rangle$, \\$\langle \Delta_j^{2/4}, \Delta_k^{1/4}, \Delta_k^{1/4} \rangle$ for $j, k = 2, 3$.
	\item $\PP^1_{2,2}$ : $\langle 1, \da{1}{2}, \da{1}{2} \rangle$, $\langle 1, \da{1}{2}, \Delta_k^{2/4} \rangle$, $\langle 1, \Delta_j^{2/4}, \Delta_k^{2/4} \rangle$ for $j, k = 2,3$.
	\item $\PP^1_{4,4}$ : $\langle 1, \Delta_j^{1/4}, \Delta_k^{3/4} \rangle$ for $j, k = 2,3$
\end{enumerate}

Again, $(b)$ and $(c)$ do not occur because of Lemma \ref{lem:nodal}.
If we denote $\bt := \sum t_{j,i} \Delta^i_j$, the genus-0 Gromov-Witten potential of $\PQ$ is written up to order of $t^3$ as
\begin{equation}\label{eq:3GWpot244}
\begin{array}{rl}
F^{\PQ}_{0} (\bt) &= \frac{1}{2} t_0^2 \log q +  \frac{1}{2} t_{0} \ta{1}{2} \ta{1}{2} \\
&+ \frac{1}{4} t_{0} \tb{2}{4} \tb{2}{4} + \tb{1}{4} \tb{3}{4} + \tc{2}{4} \tc{2}{4} + \tc{1}{4} \tc{3}{4}
+ \frac{1}{2} \ta{1}{2} ( \tb{1}{4}^2 + \tc{1}{4}^2 ) \cdot g_0 (q)\\
& + \ta{1}{2} \tb{1}{4} \tc{1}{4} \cdot g_1 (q) + \frac{1}{2} ( \tb{2}{4} \tb{1}{4}^2 + \tc{2}{4} \tc{1}{4}^2 ) \cdot g_2 (q) \\
&+ \frac{1}{2} (\tb{2}{4} \tc{1}{4}^2 + \tc{2}{4} \tb{1}{4}^2) \cdot g_3 (q) + ( \tb{2}{4} \tb{1}{4} \tc{1}{4} + \tc{2}{4} \tc{1}{4} \tb{1}{4} ) \cdot g_4 (q) + O(t^4)
\end{array}
\end{equation}

The classification in (a) shows that the domain orbi-sphere should have the same orbifold structure, that is, the contributions only come from maps $\PQ \to \PQ$.
Let $p : \C \to \PQ$ be the universal covering which factors through the $\Z_4$-quotient map $E' \to \PQ$. We abuse the notation $p$ for covering maps of both the domain and the target $\PQ$.
From the obvious symmetry between $w_2$ and $w_3$, we may fix one of the orbi-insertions by $\db{1}{4}$, similarly to what we did for $\PT$ case. So, we assume that our holomorphic orbi-sphere sends $z_2$ to $w_2$.

As before, any holomorphic orbi-sphere $u$ in our concern can be lifted to a linear map $\tilde{u} :  z \to \lambda z$ by Lemmas \ref{lem:u_orb_cover} and \ref{lem:reg}.
We set the lattice structure on $\C$ induced by the covering $\C \to \PQ$ for both the domain and the target as follows:
\begin{align*}
p^{-1} (w_1) &= \left\{ \left. \frac{1}{2} (a+ib) \right| \text{either $a$ or $b$ is an odd number, but not both.} \right\},\\
p^{-1} (w_2) &= \Z \langle 1, i \rangle \,\,(\ni 0),	\\
p^{-1} (w_3) &= \left\{ \left. \frac{1}{2} (a+ib) \right| \text{both $a$ and $b$ are odd numbers} \right\}.
\end{align*}
Here, we think of $w_2$ as the base point associated with the universal cover $(\C,0)$.
(See Figure \ref{fig:244lattice}.)

\begin{figure}
\begin{center}
\includegraphics[height=3in]{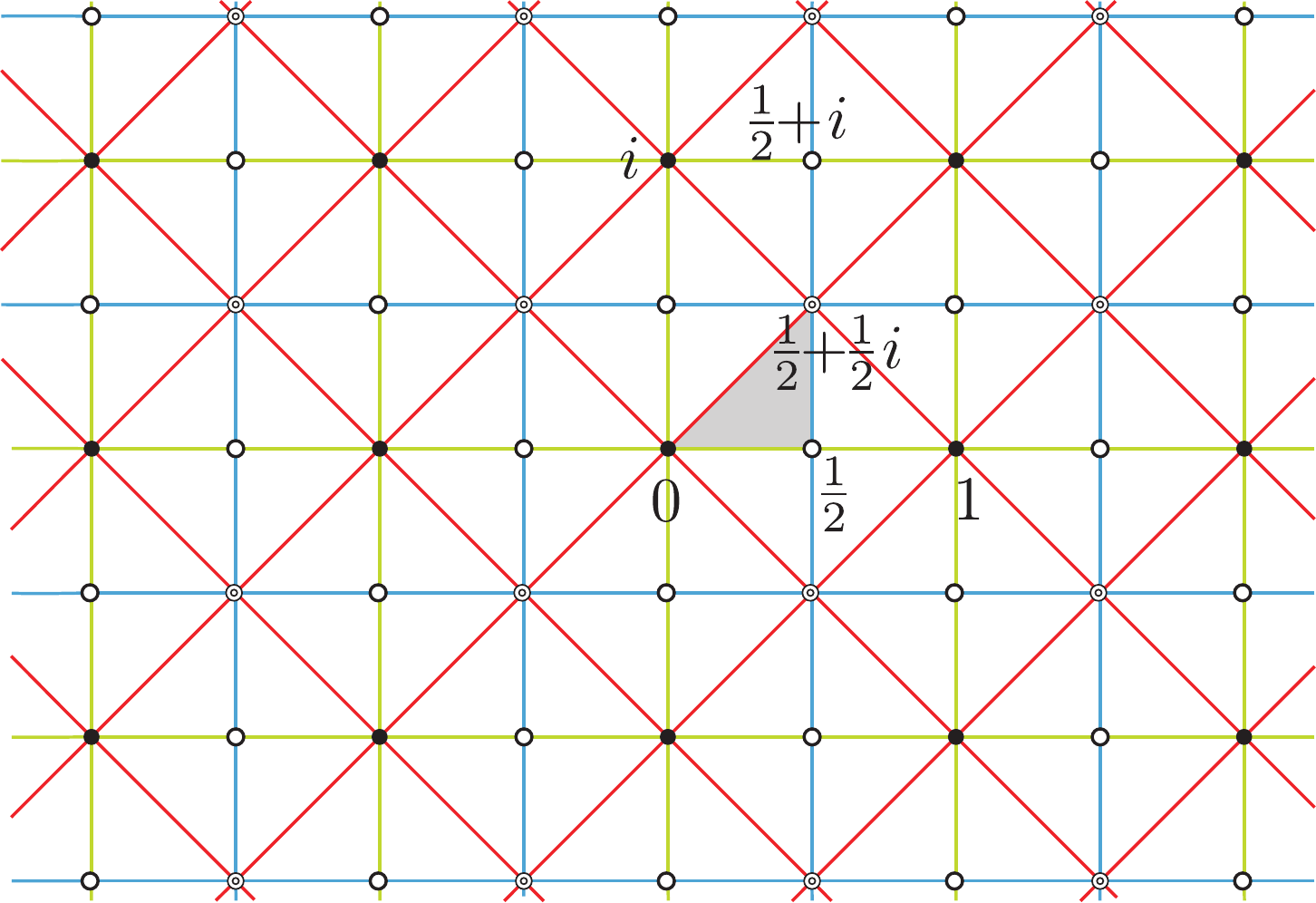}
\caption{Lattices on the universal cover of $\PQ$ : $p^{-1} (w_1)=\{ \circ \}$, $p^{-1} (w_2)= \{\bullet\}$ and $p^{-1}(w_3)=$ $\{${\tiny $\circledcirc$}$\}$}\label{fig:244lattice}
\end{center}
\end{figure}

Since we have assumed that the orbifold singular point $z_2$ is mapped to $w_2$, the lifting $\WT{u}$ of $u : \PQ \to \PQ$ maps $p^{-1} (z_2)$ to $p^{-1} (w_2)$ fixing the origin.
Therefore, $\WT{u}(z) = \lambda z$ for some $\lambda \in \Z[i]$. As mentioned, the degree of a holomorphic orbi-sphere $u$ is $|\lambda|^2 = a^2 + b^2$ if the lifting of $u$ is $\WT{u}(z) = \lambda z$ with $\lambda = a + b i$.
Let $G(q)$ denote the power series
\begin{equation}
G(q) = \sum_{a,b \in \Z} q^{a^2 + b^2}.
\end{equation}
See \eqref{eq:FandG} for first few terms of $G$.
Note that if we divide $a^2 + b^2$ by $4$, then $3$ can not appear as a remainder for any $a, b \in \Z$.
Thus, we can decompose $G$ into $G = G_{0,4} + G_{1,4} + G_{2,4}$ in accordance with the exponent of $q$ modulo $4$ .

We determine the orbi-insertion for each $\lambda = a + b i$ by the same way as before. Note that the right-angled isosceles triangle whose vertices $\{ 0, \frac{1}{2}, \frac{1+i}{2} \}$ is one of the fundamental domain of the upper-hemisphere of $\PQ$. (See the shaded region in Figure \ref{fig:244lattice}.)

Observe that the two marked points other than the origin in this fundamental domain map to
\begin{equation*}
(a + bi)\cdot \frac{1}{2} = \frac{a}{2} + \frac{b}{2}i \quad \mbox{and} \quad (a+bi)\cdot \frac{1+i}{2} = \frac{a-b}{2} + \frac{a+b}{2} i
\end{equation*}
by the linear map $z \mapsto (a+bi) z$. By proceeding as in the case of $\PH$, we see that there are only three possibilities of the type of insertions, which are listed as follows:
\begin{enumerate}[(i)]
\item
``$ \frac{a}{2} + \frac{b}{2} \in p^{-1} (w_1)$ and $\frac{a-b}{2} + \frac{a+b}{2} \in p^{-1} (w_3)$" if and only if $a^2 + b^2 \equiv 1 \mod 4$
\item ``$\frac{a}{2} + \frac{b}{2} \in p^{-1} (w_2)$ and $\frac{a-b}{2} + \frac{a+b}{2} \in p^{-1} (w_2)$" if and only if $a^2 + b^2 \equiv 0 \mod 4$
\item ``$\frac{a}{2} + \frac{b}{2} \in p^{-1} (w_3)$ and $\frac{a-b}{2} + \frac{a+b}{2} \in p^{-1} (w_2)$" if and only if $a^2 + b^2 \equiv 3 \mod 4$
\end{enumerate}
By definition of coefficients $g_i$ in \eqref{eq:3GWpot244}, holomorphic orbi-spheres with insertions (i), (ii) and (iii) precisely give rise to $g_1$, $g_2$ and $g_3$, respectively. Therefore, we conclude that $g_0 (q) = g_4 (q) =  0$, and
\begin{equation*}
\begin{array}{rl}
g_1 (q) &=  \dfrac{1}{4} \displaystyle\sum_{\substack{N =1 \\ N \equiv 1 \text{ mod } 4}}^{\infty} \sum_{\substack{a^2 + b^2 = N \\ a,b\in \Z}} q^N = \frac{1}{4} G_{1,4},\\
g_2 (q) &=  \dfrac{1}{4} \displaystyle\sum_{\substack{N =1 \\ N \equiv 0 \text{ mod } 4}}^{\infty} \sum_{\substack{a^2 + b^2 = N \\ a,b\in \Z}} q^N ,\\
g_3 (q) &=  \dfrac{1}{4} \displaystyle\sum_{\substack{N =1 \\ N \equiv 2 \text{ mod } 4}}^{\infty} \sum_{\substack{a^2 + b^2 = N \\ a,b\in \Z}} q^N.
\end{array}
\end{equation*}
Again, $\frac{1}{4}$ is due to the $\Z_4$-symmetry at the origin in the universal cover $\C$ (from the action of the local group at $w_2$) which is generated by the $i$-multiplication.

\section{Appendix : Theta series}
Recall that our results were expressed in terms of the following two power series
\begin{equation}\label{eq:FandG}
\begin{array}{rcl}
F(q) &=& \sum_{a,b \in \Z} q^{a^2 - ab + b^2} \\
&=& 1 + 6 q + 6 q^3 + 6 q^4 + 12 q^7 + 6 q^9 \\
&& + 6 q^{12} + 12 q^{13} + 6 q^{16} + 12 q^{19} + 12 q^{21} + O(q^{24}), \\
G(q) &=& \sum_{a,b \in \Z} q^{a^2 + b^2}\\
&=& 1+ 4q + 4q^2 + 4q^4 + 8 q^5 + 4 q^8 \\
&& + 4q^9 + 8 q^{10} + 8 q^{13} + 4 q^{16} + 8 q^{17} + 4 q^{18} + o(q^{20}).
\end{array}
\end{equation}
In this section, we briefly explain several number theoretic feature of $F$ and $G$. (For more details, see \cite[Chapter 4]{B} or \cite{G}.) We first provide a description of Fourier coefficients of $F$ and $G$.

\begin{prop} Write $F(q) = \sum_{N \geq 0} a_N q^N$ and $G(q) = \sum_{N \geq 0} b_N q^N$. Then
$$a_N=6 ( d_{1/3} (N) - d_{2/3} (N)),$$
$$b_N =  4 (d_{1/4} (N) - d_{3/4} (N) )$$
where $d_{j/3} (N)$ denotes the number of divisors of $N$ which are $j$ modulo $3$, and $d_{j/4} (N)$ the number of divisors of $N$ which are $j$ modulo $4$
\end{prop}

\begin{proof}
We  only prove the first identity for $a_N$, and refer readers to \cite[Theorem 3]{G} for $G$. The following is a simple modification of the argument given in \cite{G}, but we repeat it here for completeness.

Recall that for $\tau= e^{2 \pi \sqrt{-1} / 3}$, $|a + b \tau|^2 = a^2 - ab + b^2$. This gives a structure of Euclidean domain in $\Z [ \tau]$ (this ring is usually called the ring of {\em Eisenstein integers} or {\em Eulerian integers}). In particular, $\Z [\tau]$ is a unique factorization domain, and hence a prime factorization in this ring makes sense up to units which are $=\{\pm1, \pm \tau, \pm \tau^2\} = \{ (1+\tau)^k \, | \, 0 \leq k \leq 5\}$ (and also up to the order of factors). It is known that a prime number in $\Z [\tau]$ is either a prime number in $\Z$ which is $2$ modulo $3$, or $a+b \tau$ whose modulus square $|a+b \tau|^2$ is a prime number in $\Z$. In latter case, $|a+b\tau|^2$ is alway $1$ modulo $3$ unless it is $3=(1-\tau)\overline{(1-\tau)}=(1-\tau)(2+\tau)$ itself. (Of course, a prime number multiplied with a unit is also prime.)

Note that finding solutions of 
\begin{equation}\label{eq:seeksol}
a^2 -ab + b^2 = (a+ b \tau) \overline{(a+b \tau)}  = N \in \Z
\end{equation}
is equivalent to finding factorizations $N=\alpha \beta$ of $N$ in $\Z[\tau]$ such that $\beta = \bar{\alpha}$ where $\bar{\alpha}$ is the complex conjugation of $\alpha$. Let $N=3^f n_1 n_2$ with $n_1 = \prod_{p \equiv1 ({\rm mod} 3)} p^r$ and $n_2 = \prod_{q \equiv 2 ({\rm mod} 3)} q^s$. Then the prime factorization of $N$ in $\Z [\tau]$ can be written as
$$N= \{ (1-\tau)(2+\tau) \}^f \prod_{\substack{c^2 -  cd + d^2 =p \\ p\equiv1 ({\rm mod} 3)}} \{(c+ d \tau) \overline{(c+d \tau)}\}^r \prod_{q \equiv 2 ({\rm mod} 3)} q^s$$
where $c+ d \tau$ and $\overline{c + d \tau}$ come in pair for each $p$ since $N$ is an integer. Now the condition $\beta=\bar{\alpha}$ forces them to be of the following forms:
\begin{equation}\label{eq:alphabetaufd}
\begin{array}{rl}
\alpha&= (1+ \tau)^{t} (1-\tau)^{f_1} (2+\tau)^{f_2} \prod  \{(c+ d \tau)^{r_1} \overline{(c+d \tau)}^{r_2}\} \prod q^{s_1} \\
\beta&= (1+\tau)^{-t} (1-\tau)^{f_2} (2+ \tau)^{f_1} \prod  \{(c+ d \tau)^{r_2} \overline{(c+d \tau)}^{r_1}\} \prod q^{s_2}
\end{array}
\end{equation}
with $0\leq t \leq 5$, $f_1 + f_2 =f$, $r_1 +r_2 = r$ and $s_1 +s_2=s$. $\bar{\beta} = \alpha$ also implies that $s_1 = s_2$, so there is no solution to \eqref{eq:seeksol} if $s$ is odd. Let us assume that $s$ is even from now on. Then $s_i$'s are uniquely determined (as the half of $s$).  Observe that $t$ has six choices and $f_1$, $r_1$ determine $f_2=f-f_1$, $r_2=r-r_1$ respectively. Thus, there are seemingly $6 (f+1) \prod (r+1)$ number of choice for $\alpha$ and $\beta$ satisfying \eqref{eq:alphabetaufd}. However, replacing one $(1-\tau)$ by $(2+\tau)$ for in the expression of $\alpha$ \eqref{eq:alphabetaufd} affects $\alpha$ by multiplying a unit since $(2+\tau)/(1-\tau) = 1 +\tau $ is a multiplicative generator of the group of units in $\Z [\tau]$. Getting rid of this redundancy, the number of pairs $(\alpha,\beta)$ satisfying $N= \alpha \beta$ and $\beta=\bar{\alpha}$ is given by $6 \prod (r+1)$. It is easy to check that this number is same to $6 ( d_{1/3} (N) - d_{2/3} (N) )$.
\end{proof}

\begin{remark}
From the proof, we see that ``$6$" in  the expression of $a_N$ is related to the number of units in the ring $\Z [\tau]$, which give the symmetries on the associated moduli space of orbi-spheres (see the last paragraph of Section \eqref{subsec:symm333}).
\end{remark}

We next describe $F$ and $G$ in terms of famous Jacobi theta functions.
The definitions of related Jacobi theta functions are given as follows.
\begin{definition}
The second and the third Jacobi theta functions are the power series $\theta_2$ and $\theta_3$ in $q$ which are defined as follows:
$$\theta_2(q) := \sum_{-\infty}^{\infty} q^{(n+1/2)^2},$$
$$\theta_3(q) := \sum_{-\infty}^{\infty} q^{n^2}.$$
\end{definition}

\begin{remark}
Originally, theta functions are two variable functions depending on $z$ and $q$. Above $\theta_i$ is indeed obtained by putting $z=0$.
\end{remark}
%
%

Let us now express $F(q)$ and $G(q)$ in terms of $\theta_i$'s ($i=2,3$). Firstly for $F$, observe that the number of integer solutions of $x^2 - xy + y^2 = N$ is equivalent to that of solutions of $(m^2  + 3n^2)/4 =N$. To see this, simply put $m=x+y$ and $n=x-y$ to $(m^2 + 3n^2)/4$. Note that $m$ and $n$ should have the same parity. Therefore,
\begin{eqnarray*}
F(q) &=& \sum_{x,y \in \Z} q^{x^2 - xy + y^2} \\
&=& \sum_{m,n : \,\,even } q^{\frac{m^2 + 3n^2}{4}}  + \sum_{m,n : \,\, odd } q^{\frac{m^2 + 3n^2}{4}} \\
&=& \sum_{k,l \in \Z} q^{k^2 + 3l^2} + \sum_{k,l \in \Z} q^{ ((k+1)/2)^2 + 3( (l+1)/2)^2} \\
&=& \theta_3 (q) \theta_3 (q^3) + \theta_2 (q) \theta_2 (q^3) 
\end{eqnarray*}
The expression of $G(q)$ is even simpler since
$$ G(q) = \sum_{x,y \in \Z} q^{x^2 + y^2} = (\theta_3 (q))^2.
$$

In general, the theta function associated with a binary quadratic form $Q(x,y)=ax^2 +bxy + cz^2$ is defined by
$$ \theta_Q (z)= \sum_{(x,y)\in \Z^2} \exp(2\pi izQ(x,y)) $$
where we have used the substitution $q = \exp({2 \pi i z})$ mostly in the paper. 
In the Fourier expansion,
$$\theta_Q (z) = \sum_{N=0}^\infty R_Q (N) \exp(2\pi i N z),$$
the numbers $R_Q (N)$ are called the representation numbers of the form $Q$, and hence $a_N$ and $b_N$ above are given as $R_F (N)$ and $R_G (N)$, respectively.

These theta functions are known to be modular forms of weight $1$ on (an appropriately defined subgroup of) the modular group. We believe that this modularity of $F$ and $G$ may help to compare our result with the one given in \cite{ST} or \cite{MR}.

\bibliographystyle{amsalpha}

\end{document}